\documentclass[preprint,12pt]{elsarticle}

\usepackage{lipsum}

\usepackage{tikz-cd}
\usepackage{amsmath}
\usepackage{amssymb}
\usepackage{amsthm, thmtools}
\usepackage{mathtools}
\usepackage{mathrsfs}

\usepackage{bbm}
\usepackage{bm}
\usepackage{MnSymbol}   
\usepackage{cases}
\usepackage[normalem]{ulem}

\usepackage[width=.80\textwidth,font=footnotesize]{caption}
\usepackage[letterpaper, margin=1in]{geometry}
\usepackage[title]{appendix}
\usepackage{fancyhdr}
\usepackage{caption}
\usepackage{subcaption}

\usepackage[shortlabels]{enumitem}
\usepackage[ddmmyyyy]{datetime}
\usepackage{hhline}

\usepackage{listings}
\newcommand{\N}{\mathbb{N}}

\newcommand{\R}{\mathbb{R}}

\newcommand{\T}{\mathbb{T}}

\newcommand{\pr}{\mathbb{P}}

\DeclarePairedDelimiter\abs{|}{|}
\DeclarePairedDelimiterX{\inner}[2]{\langle}{\rangle}{#1,#2}
\DeclarePairedDelimiterX{\norm}[1]{\|}{\|}{#1}

\DeclareMathOperator{\ExpOp}{\mathbb{E}}
\newcommand{\E}[1]{\ExpOp\left[#1\right]}

\DeclarePairedDelimiter\floor{\lfloor}{\rfloor}

\newtheorem{innercustomthm}{Theorem}
\newenvironment{customthm}[1]
  {\renewcommand\theinnercustomthm{#1}\innercustomthm}
  {\endinnercustomthm}

\newtheorem{thm}{Theorem}[section]
\newtheorem*{thm*}{Theorem}
\newtheorem{lemma}[thm]{Lemma}	\newtheorem*{lemma*}{Lemma}
\newtheorem{cor}[thm]{Corollary}
\newtheorem*{cor*}{Corollary}

\theoremstyle{definition}	
\newtheorem{Def}[thm]{Definition}

\theoremstyle{definition}

\theoremstyle{plain}		
\theoremstyle{definition}	\newtheorem{rmk}[thm]{Remark}

\theoremstyle{definition}	

\theoremstyle{definition}

\begin{document}
\begin{center}
{\Large The Isochronal Phase of Stochastic PDE and Integral Equations: \\ Metastability and Other Properties } \\
{\textsc{ (preprint)} }
\end{center}

\begin{center}
\begin{minipage}[t]{.75\textwidth}
\begin{center}
Zachary P.~Adams \\
Max Planck Institute for Mathematics in the Sciences\\
zachary.adams@mis.mpg.de\\
$\,$\\
James MacLaurin\\
New Jersey Institute of Technology \\
maclauri@njit.edu\\ 
\hfill
\end{center}
\end{minipage}
\end{center}
\begin{center}
\today 
\end{center}
\vspace{.5cm}
\setlength{\unitlength}{1in}

\vskip.15in

{
\centerline
{\large \bfseries \scshape Abstract}
We study the dynamics of waves, oscillations, and other spatio-temporal patterns in stochastic evolution systems, including SPDE and stochastic integral equations. 
Representing a given pattern as a smooth, stable invariant manifold of the deterministic dynamics, we reduce the stochastic dynamics to a finite dimensional SDE on this manifold using the isochronal phase. 
The isochronal phase is defined by mapping a neighbourbhood of the manifold onto the manifold itself, analogous to the isochronal phase defined for finite-dimensional oscillators by A.T.~Winfree and J.~Guckenheimer. 
We then determine a probability measure that indicates the average position of the stochastic perturbation of the pattern/wave as it wanders over the manifold. 
It is proved that this probability measure is accurate on time-scales greater than $O(\sigma^{-2})$, but less than $O(\exp(C\sigma^{-2}))$, where $\sigma\ll1$ is the amplitude of the stochastic perturbation. 
Moreover, using this measure, we determine the expected velocity of the difference between the deterministic and stochastic motion on the manifold. 
}\\


\section{Introduction}
\label{sec:Intro}
This paper presents results on stochastic perturbations of spatio-temporal patterns in stochastic PDE and integral equations. 
Specifically, we determine a probability distribution that characterizes the long-time behaviour of certain patterns that are stationary or periodic in time under small noise perturbations. 
These results are obtained using an abstract framework that encompasses travelling waves in nerve axons \cite{HH52,FH61}, wandering bumps in neural field equations \cite{B12}, and spiral waves in excitable media \cite{J03,LJL08,UWT10}. 
Our results pave the way for the quantitative characterization of the long-time behaviour of transient coherent structures in evolution systems perturbed by small noise. 

Our mathematical framework is as follows. 
We study systems of $n$ stochastic evolution equations of the form 
\begin{equation}
\label{eq:SDE}
dX_t\,=\, \left(AX_t + N(X_t)\right)\,dt + \sigma B(X_t)\,dW
\end{equation}
driven by noise of amplitude $\sigma$ which is white in time, and may be white or coloured in space. 
Solutions to \eqref{eq:SDE}~take values in a Hilbert space $H$ of functions mapping a $d$-dimensional spatial domain $O$ to $\R^n$. 
We take $O$ to be a periodic domain, such as the torus $\T$ or a sphere $\mathbb{S}^d$ for some $d\in\N$. 
The operator $A$ is a linear operator modelling, for instance, diffusion, while $N$ is a  nonlinear operator modelling local and non-local reactions. 
Further technical assumptions on \eqref{eq:SDE}~and remarks on the solution theory we work with are stated in Section \ref{sec:Setup}.  
We emphasize here however that we do not require the noise in \eqref{eq:SDE}~to be trace class, 

As noted in the paragraphs above, \eqref{eq:SDE}~includes models of various real world systems, though the two classes of systems in which we are primarily interested in here are (i) reaction-diffusion systems (specifically, nerve axon equations, in the sense of Evans \cite{E72i}), in which case $A=\Delta$ is the Laplace operator and $N$ is a polynomial, and (ii) neural field equations \cite{B12}, in which case $Lx = -a^{-1}x$ for a constant $a>0$, and $N$ is an integral operator. 
Further discussion of these examples and how they fit into our framework can be found in Remark \ref{rmk:EX}. 

We denote the solution process of \eqref{eq:SDE}~by $X_t$. 
Our particular interest is in the behaviour of \eqref{eq:SDE}~near invariant manifolds of the system without noise. 
Specifically, we assume that the corresponding deterministic system, 
\begin{equation}
\label{eq:ODE}
\partial_tx\,=\, Ax + N(x)\,\eqqcolon\,V(x), 
\end{equation}
possesses a stable invariant manifold\footnote{Specifically, a stable normally hyperbolic invariant manifold; see \cite{BLZ00}~for a precise definition}~$\Gamma\subset H$, which may be parameterized by a finite dimensional manifold $\mathcal{S}$. 
In concrete terms, the manifold $\Gamma$ consists of all possible states of a particular spatio-temporal pattern which persists in \eqref{eq:ODE}~for all time. The spatio-temporal patterns in which we are principally interested include periodic solutions, such as travelling pulses on periodic spatial domains \cite{AK15, KSS97}~and spiral waves \cite{BL08,SSW97,X93}. 
In the former case, we take $\mathcal{S}$ to be the circle $\mathbb{S}^1$, while in the latter case we take $\mathcal{S}$ to be the special Euclidean group $SE(O)$. 
Our approach also applies to stationary solutions on a spherical spatial domain, such as those occurring in certain neural field equations \cite{CGG01,D02,VRFC17}. 
We anticipate that our framework may be extended to include other time-periodic patterns such as breathers \cite{CW21, FG08}. 

Generally, it is very difficult to study stochastic perturbations of the patterns modelled by the manifold $\Gamma$. 
To facilitate such a study, it is common to ``reduce'' the infinite dimensional dynamics of \eqref{eq:SDE}~to a finite dimensional stochastic process on $\mathcal{S}$ \cite{BLZ00,engel2021random}. 
In this work, we reduce the dynamics of \eqref{eq:SDE}~near the deterministic invariant manifold $\Gamma$ by using the \emph{isochronal phase}. 
The term isochronal phase is commonly attributed to Winfree \cite{W74}, where it was used in the context of oscillatory dynamics in finite dimensional biological systems. 
Winfree's isochronal phase was then studied rigorously by Guckenheimer \cite{G75}, and has since been used in many contexts, as described in Section \ref{sec:Literature}. 

Towards the end of defining the isochronal phase, we let 
\[
(\alpha\mapsto\gamma_\alpha): \mathcal{S}\rightarrow\Gamma, 
\]
denote the parameterization of $\Gamma$ by $\mathcal{S}$. 
The definition of the isochronal phase depends on \eqref{eq:ODE}~having a well-posed solution theory, along with the stability of $\Gamma$. 
In this case, for some $\delta$-neighbourhood $\Gamma_\delta$ of $\Gamma$ there exists a map $\pi:\Gamma_\delta\rightarrow\mathcal{S}$ such that for each $x\in\Gamma_\delta$, $\pi(x)$ is the unique point in $\mathcal{S}$ such that $x$ and $\gamma_{\pi(x)}$ are attracted to each other under the flow of \eqref{eq:ODE}. 

The map $\pi:\Gamma_\delta\rightarrow\mathcal{S}$, henceforth referred to as the \emph{isochron map}, allows us to reduce the dynamics of \eqref{eq:SDE}~near $\Gamma$ to a finite dimensional stochastic process in a natural way. 
In particular, we define the \emph{isochronal phase}~of \eqref{eq:SDE}~near $\Gamma$ as $\pi_t\coloneqq\pi(X_t)$. 
In many cases (for instance in the case of additive white noise, when $B\equiv1$ and $W$ is white in space and time), the noise in \eqref{eq:SDE}~will cause the solution to exit $\Gamma_\delta$ at some almost surely finite stopping time, denoted by $\tau_\delta$. 
This means that the isochronal phase $\pi_t$ is only guaranteed to be defined for $t<\tau_\delta$. 
Nevertheless, when the noise amplitude $\sigma$ of \eqref{eq:SDE}~is small, one expects $\tau_\delta$ to be large with high probability, due to the deterministic stability of $\Gamma$.  
Indeed, it is proven in \cite{M22}~that $\tau_\delta$ has a high probability of being greater than $\exp(C\sigma^{-2}\delta^2)$ for some $C>0$.  
One might therefore expect to observe some sort of metastable behaviour of the isochronal phase on the random time interval $[0,\tau_\delta)$. 

The remainder of the paper is structured as follows. 
Our major results are summarized in Section \ref{sec:Results}, after which follows a non-exhaustive survey of related literature in Section \ref{sec:Literature}. 
Following this, we state the precise setup in which we work in Section \ref{sec:Setup}, also discussing systems that fall into our framework and the limitations of our study. 
Section \ref{sec:Regularity}~then proves various regularity properties of the isochron map. 
In Section \ref{sec:Ergodic}, these regularity properties are leveraged to obtain results on the long time behaviour of \eqref{eq:SDE}~near $\Gamma$~in the small noise regime.

\subsection{Results}
\label{sec:Results}
In this work, we prove two major results regarding the isochronal phase of \eqref{eq:SDE}~under the assumptions of Section \ref{sec:Setup}. 
As discussed in Section \ref{sec:Setup}, the examples of travelling waves in nerve axon equations, wandering bumps in neural field equations, and spiral waves in reaction-diffusion systems fit into our framework when analyzed on periodic spatial domains. 
The first result, proven in Theorem \ref{thm:Ito}~and the discussion of Section \ref{sec:Regularity}, provides an It{\^o}~formula for the evolution of $\pi_t$.

\begin{customthm}{A}
The isochron map $\pi:\Gamma_\delta\rightarrow\mathcal{S}$ is twice continuously differentiable \footnote{In the topology of a Banach subspace of $H$, specified in Section \ref{sec:Setup}. }.
Moreover, $\pi_t=\pi(X_t)$ satisfies the It{\^o}~formula \eqref{eq:piIto}. 
\end{customthm}
 
In the case of noise that is uncorrelated in space, the above theorem is non-trivial, and requires us to prove a few technical lemmas on the regularity of the isochron map $\pi:\Gamma_\delta\rightarrow\mathcal{S}$. 
The It{\^o}~formula and regularity properties of the isochron map that we prove in Section \ref{sec:Regularity}~are essential to the remainder of our study, and allow us to prove a metastability result for the isochronal phase via relatively direct arguments. 

To understand the long time behaviour of the isochronal phase, we first need estimates on the exit time $\tau_\delta$. 
Indeed, such estimates follow immediately from the result of the second author in \cite{M22}~based on the theory of Kapitula \&~Promislaw \cite{KP13}. 
The result of \cite{M22}~holds for many of the systems in which we are interested. 
Hence, we assume that a concentration inequality of the following form holds. 

\begin{lemma*}
There exist $c,C>0$ such that for all $t>0$ and small $\sigma,\delta>0$, 
\begin{equation}
\label{eq:Concentration}
\pr\left(t\ge\tau_\delta\right)\,\le\,ct\exp\left(-C\sigma^{-2}\delta^2\right). 
\end{equation}
\end{lemma*}

When such a concentration inequality holds, we are able to prove the main result of this paper. 
The most significant result of this paper establishes the existence of a probability measure $P_*$ on $\mathcal{S}$ which behaves as a sort of ergodic measure for the isochronal phase $\pi_t$ on timescales greater than $\sigma^{-2}$, but less than $\exp\left(C\sigma^{-2}\delta^2\right)$ for some $C>0$. The measure $P_*$ is constructed as the invariant measure of a Markov process on $\mathcal{S}$ that approximates $\pi_t$ on finite time intervals.  
Hence, we may characterize $P_*$ as the zero measure of an elliptic operator $\mathcal{A}$ on $\mathcal{S}$, given by the Markov generator of this approximation of the isochronal phase. 
For further details and an expression for the operator $\mathcal{A}$, see Theorem \ref{thm:ApproximateErgodic}.

\begin{customthm}{B}
There exists a measure $P_*$ on $\mathcal{S}$, constructed in Theorem \ref{thm:ApproximateErgodic}, such that the following holds. 
Taking small $\epsilon,\delta>0$, there exists $C>0$ such that, if $(t_{\sigma})_{\sigma>0}$ is a family of times satisfying 
\begin{equation}
\label{eq:t_sigma}
\sigma^2t_\sigma\,\xrightarrow[\sigma\rightarrow0]{}\,\infty\qquad\text{ and }\qquad 
t_\sigma\,\le\,\exp\left(C\sigma^{-2}\delta^2\right) \quad\text{for small $\sigma>0$,} 
\end{equation} 
then for $g\in C^2(\mathcal{S})$, we have 
\[
\pr\left(\abs*{\frac{1}{t_\sigma}\int_0^{t_\sigma}g(\pi_s)\,ds - P_*(g)}<\epsilon,\,t_\sigma<\tau_\delta\right)\,\xrightarrow[\sigma\rightarrow0]{}\,1. 
\]
\end{customthm}

From these results, we are then able to characterize the noise-induced drift of stochastically-perturbed limit cycles. 
The following result is proven as Corollary \ref{cor:Spheres}~in Section \ref{sec:Spheres}. 
Note that the following result is a generalization to an infinite-dimensional-setting of a result in Giacomin, Poquet, \&~Shapira \cite{GPS18}. 

\begin{cor*}
Suppose that $\mathcal{S}=\mathbb{S}^1$, and let $\overline{\pi}_t^\sigma \in \mathbb{R}$ be the arc-length of the isochronal phase of \eqref{eq:SDE}~(so that if we identify $\mathcal{S}^1$ with the interval $[0,2\pi]$ in the usual way, then $\pi_t=\overline{\pi}_t^\sigma\mod2\pi$ and $t \mapsto \overline{\pi}^{\sigma}_t$ is continuous). There exists a vector field $\mathcal{V}$ on $\mathcal{S}$ such that if $(t_\sigma)_{\sigma>0}$ is as in \eqref{eq:t_sigma}, then for all small  $\epsilon,\delta>0$ we have 
\begin{equation}
\label{eq:CorShift}
\pr\left(\abs*{\frac{1}{\sigma^2t_\sigma}\left(\overline{\pi}^\sigma_{t_\sigma} -\overline{\pi}_{t_\sigma}^0\right)- \mathbb{E}^{P_*}\left[ \mathcal{V}\right]}<\epsilon,t_\sigma<\tau_\delta\right)\,\xrightarrow[\sigma\rightarrow0]{}\,1. 
\end{equation}
\end{cor*}

The precise assumptions that allow us to conclude the above results are described in Section \ref{sec:Setup}. 
First, we review the literature leading to our conclusions.

\subsection{Literature}
\label{sec:Literature}
This section is a non-exhaustive review of methods and results related to those used and obtained in this paper. 

We first remark on the relevance of our results to applied disciplines. 
In neural field theory, there has been a long-standing interest in understanding the wandering of patterns over long periods of time due to fluctuations in neural activity \cite{BW12,del2016adaptation,kilpatrick2013wandering,maclaurin2020wandering,carrillo2022noise}. 
In fact, there is an ongoing debate concerning the degree to which the wandering of ``bumps'' of excited activity in the visual cortex is due to a slow adaptation of the neural connections, as opposed to stochastic fluctuations induced by anomalous inputs from other areas of the brain \cite{del2016adaptation}. The results of this paper provide a means of computing 
estimates for the probability distribution of the bump's position as it wanders.
Similarly, the deviation in the speed of travelling waves in scalar reaction-diffusion equations has been under investigation since at least the 90s \cite{garcia2012noise,MS95}, though the question has only recently been addressed in systems of equations and for pulse type travelling waves, rather than just fronts in single equations \cite{EGK20,K20}. 

There is a well-established literature on the reduction of stochastic partial differential equations to finite-dimensional systems \cite{duan2014effective}. In the mathematical literature, most of the works influencing the present study reduce the dynamics of \eqref{eq:SDE}~to the \emph{variational phase}, rather than the isochronal phase. 
The variational phase is a map $\beta :\Gamma_\delta\rightarrow\mathcal{S}$ defined by minimizing $x\mapsto\norm*{x-\gamma_{\beta(x)}}_0$ for some semi-norm $\norm*{\cdot}_0$ on $H$ (a more explicit expression of a particular variational phase is given in Appendix \ref{sec:IsoVar}). 
It is well-known that as long as $x$ is close enough to the stable manifold, $\beta$ is always uniquely defined \cite[Lemma 4.3.3]{KP13}.

The variational phase was first developed for the study of orbital stability of traveling waves and patterns in deterministic systems -- see for instance \cite{cazenave1982orbital,pego1994asymptotic}, and the reviews in \cite{bievre2015orbital,KP13,volpert1994traveling}. 
Many researchers have worked to generalize the theory of orbital stability to a stochastic setting \cite{armero1998ballistic,BW12,LGNS04,IM16,M22}. 
Later, the variational phase was used in the context of tracking the position of travelling pulses in neural field equations by Inglis \&~MacLaurin \cite{IM16}, Lang \cite{L16}, and Lang \&~Stannat \cite{LS16}. 
Meanwhile, in reaction-diffusion systems, the variational phase has been used by Inglis \&~MacLaurin \cite{IM16}, Eichinger, Gnann, \&~Kuehn \cite{EGK20}, Kr{\"u}ger \&~Stannat  \cite{KS14}, \cite{KS17}, and Stannat \cite{S13}~to study travelling pulses. 
The variational phase has also been used by Kuehn, MacLaurin, \&~Zucal \cite{KMZ22}~to study stochastic perturbations of rotating spirals in reaction-diffusion systems. 

In the works \cite{IM16}, \cite{M22} and \cite{KMZ22} the time up to which the variational phase is well-defined is estimated by studying the time that it takes the system to escape the attracting manifold. Theorem B, above, translates this result to the isochronal phase using the arguments of Appendix \ref{sec:IsoVar}. 

A novel phase reduction method, similar in many respects to the variational phase, but distinct in a few technicalities, was developed by Hamster \&~Hupkes \cite{HH19, HH20a, HH20b}. 
Using this phase, Hamster \cite{H20}~makes predictions about the asymptotic speed of stochastic travelling waves in the  Fitzhugh-Nagumo system, though this speed is not proven to exist in any sense. 
The predictions of \cite{H20}~are in line with our Theorem \ref{thm:ApproximateErgodic}, which should hold for these waves on sufficiently large periodic spatial domains. 

Returning to the isochronal phase introduced in \cite{G75,W74}, we remark that it has been used in many past works to study the effect of noise on finite-dimensional oscillators \cite{wilson2016isostable}. 
For instance, the isochronal phase has been used to understand phenomena such as stochastic synchronization \cite{pikovsky2000phase,teramae2009stochastic}. 
Giacomin, Poquet, \&~Shapira \cite{GPS18}~obtain accurate estimates for the long time phase-shift of a general oscillatory system perturbed by small noise. 
Indeed our Corollary \ref{cor:Spheres} can be seen as a generalization of their Theorem 2.6 to an infinite dimensionsal SDE. 
Cao, Lindner, \&~Thomas \cite{CLT20}~and Cheng \&~Qian \cite{CQ21}~similarly study the long time behaviour of stochastic oscillators using the isochronal phase in the small noise regime, and obtain results in line with the present study.  See also the first author's work \cite{A22}, where results similar to \cite{GPS18}~were obtained using the theory of quasi-ergodic measures \cite{CMSM13}. 

Other studies which use the isochronal phase to study spatio-temporal patterns in SPDE include Bertini, Brassesco, \&~Butt{\`a}~\cite{BBB08}, Brassesco, De Masi, \&~Presutti \cite{BDMP95}~and Brassesco, Butt{\`a}, De Masi, \&~Presutti \cite{BBDMP98}. Nakao \cite{nakao2014phase} employs a linearization of the isochronal phase to study the effect of noise on spatially-extended oscillations over finite timescales. Recently, Lucon and Poquet \cite{luccon2022existence,luccon2021periodicity} have performed an extensive analysis of mean-field interacting particle systems in the large size limit. They obtain a phase reduction (analagous to the one in this paper), and also obtain long-time ergodicity results that characterize the drift of the empirical measure.

Parallel to the study of the isochronal \&~variational phase is a theory for the projection of SPDE onto slow manifolds. 
This theory has been developed by Funaki, Bl{\"o}mker, Hairer and co-workers \cite{Bl2008,blomker2005amplitude,funaki1995scaling}~(see also the work \cite{katzenberger1990solutions}, that projects a finite-dimensional SDE onto a slow manifold). 
For applications of this technique in biology, see the review by Parsons \&~Rogers \cite{parsons2017dimension}.

We also refer to alternative phase reduction methods developed in the physics literature, usually for finite dimensional oscillators, though the definitions could easily be lifted to a Hilbert space setting. 
For instance, \cite{B17, teramae2009stochastic, YA07}, define phase maps by referring to the unperturbed (deterministic) oscillator. 
Meanwhile, \cite{SP10a, SP13, SPKR12}, attempt to account for the noise-induced fluctuations of an oscillator when defining a phase map. 
Few of the phase reduction methods developed in these studies are suitable for addressing stability or ergodicity properties of stochastically perturbed patterns. 

One might also compare the results we obtain in \eqref{eq:SDE}~under the assumptions of Section \ref{sec:Setup}~with the study of travelling fronts in stochastic FKPP equations \cite{MS95}. 
In these systems, Mueller, Mytnik, \&~Quastel \cite{MMQ11}~proved that the addition of small amplitude noise has a large effect on the asymptotic speed of travelling waves in FKPP type equations, verifying the so called Brunet-Derrida conjecture of \cite{BD01}. 
See also Mueller, Mytnik, \&~Rhyzik \cite{MMR21}. 
However, the techniques used to study these equations do not extend to the systems we study, based as they are on a maximum principal that does not hold for \eqref{eq:SDE}~in general. 
Conversely, our techniques do not presently extend to travelling fronts in stochastic FKPP equations, as we require the diffusion coefficient $B$ in \eqref{eq:SDE}~to be differentiable. 

Finally, we remark that our methods may apply to the study of the metastability of breathers under perturbation. 
Breathers are known to occur in neural field equations~\cite{F04}~and both continuous and discrete nonlinear Schr{\"o}dinger equations \cite{FG08}. 
In Caballero \&~Wayne \cite{CW21}~and Eckmann \&~Wayne \cite{EW18, EW20}, the metastability of these breathers when perturbed by a damping force has been studied. 
Indeed, in \cite{CW21}~it has been conjectured that perturbations by damping are formally related to perturbations by noise. 
However, applying our results to this setting would require further work.

\section{Setup} 
\label{sec:Setup}
We now state the assumptions imposed on \eqref{eq:SDE} and \eqref{eq:ODE}. 
Throughout the document, we consider the dynamics of \eqref{eq:ODE}~to take place on a spatial domain $O$, which in many applications is bounded or periodic. 
Unique, global-in-time solutions to \eqref{eq:ODE}~are known to exist in a Hilbert space $H$ of functions $O\rightarrow\R^n$ under the following assumptions. The norm and inner product on $H$ are written as (respectively) $\norm{\cdot}$ and $\langle \cdot, \cdot \rangle$.

\begin{enumerate}[{\bf A.}]
\item $N$ is a nonlinearity defined on some $E\subset H$ such that $E$ is a Banach space with norm $\norm*{\cdot}_E$, and the embedding $E\hookrightarrow H$ is continuous. 
We moreover assume that $N$ is three-times continuously Fr{\'e}chet differentiable in the topology of $E$.  That is, for any $x\in E$ we denote by $DN(x): E \to E$, $D^2N(x):  E \times E \to E$ and $D^3 N(x):E \times E \times E \to E$  its (respective) first, second and third Fr{\'e}chet derivatives at $x$, and these are all assumed to be bounded operators. 
Assume that there exists an orthonormal basis $\{e_k\}_{k\in\N}$ of $H$ such that $e_k\in E$ for each $k\in\N$. 
\item $A$ is a linear operator with a domain of definition $D(A) \subset E$ that is dense in $H$, and generates a $C_0$-semigroup $(\Lambda_t)_{t\ge0}$ on $H$. 
Moreover, letting $\norm*{\cdot}$ denote the norm of $H$, there exists $\omega<0$ such that for all $t\ge0$, 
\begin{equation}
\label{eq:Lambda}
\max\left\{\norm*{\Lambda_t},\, \norm*{\Lambda_t}_E\right\}\,\le\,e^{\omega t}. 
\end{equation}
\end{enumerate}
We wish to avoid unnecessary technicalities resulting from the differential geometry. Indeed for almost every application of this theory that we can think of, the stable manifold $\Gamma$ is very regular (it usually corresponds to a continuous group action, such as translation and / or rotation \cite{KP13}). We will thus assume that $\Gamma$ can be parameterized by a compact subset $\mathcal{S}$ of $\mathbb{R}^m$. In this way, when we describe the induced stochastic dynamics on $\mathcal{S}$, we can avoid introducing co-ordinate charts and we can instead simply consider the induced phase dynamics as a $\mathbb{R}^m$-valued stochastic process.

\begin{enumerate}[{\bf C.}] 
\item We assume that the deterministic system \eqref{eq:ODE}~has a stable normally hyperbolic invariant manifold $\Gamma \subset E$ that can be smoothly parameterized by a smooth finite dimensional manifold $\mathcal{S}\subset\R^m$, where $m\in\N$. We write $\Gamma=\{\gamma_\alpha\}_{\alpha\in\mathcal{S}}$. $\Gamma$ is assumed to be in the domain of $A$. We assume that the tangent space $T_{\alpha}\mathcal{S} \subset \mathbb{R}^m$ of $\mathcal{S}$ is $(m-1)$-dimensional at every point $\alpha$. Let $\lbrace \mathfrak{v}_{\alpha,j} \rbrace_{1\leq j \leq m} \subset \mathbb{R}^m$ be an orthonormal basis of $\mathbb{R}^m$ such that the first $(m-1)$ vectors span $T_{\alpha}\mathcal{S} $. That is, writing $ \mathfrak{v}_{\alpha,j} = ( \mathfrak{v}^a_{\alpha,j})_{1\leq a \leq m}$,
\begin{equation}
T_{\alpha}\mathcal{S} = \big\lbrace y \in \mathbb{R}^m : \sum_{a=1}^m y^a \mathfrak{v}^a_{\alpha,m} = 0 \big\rbrace . \label{eq: tangent hyperplane}
\end{equation}
Suppose that $( \alpha_{\epsilon})_{\epsilon \geq 0} \in \mathcal{S}$ is such that $\lim_{\epsilon \to 0^+}\epsilon^{-1}\lbrace \alpha_{\epsilon} - \alpha \rbrace = \mathfrak{v}_{\alpha,i}$. Then write $\psi_{\alpha}^i \in H$ to be such that
\begin{align}
\psi_{\alpha}^i = \lim_{\epsilon \to 0^+} \epsilon^{-1}\lbrace \gamma_{\alpha_{\epsilon}  } - \gamma_{\alpha} \rbrace
\end{align}
The smoothness assumptions imply that all higher-order derivatives exist as well.
\end{enumerate}

In order that we may precisely specify the stability properties of the manifold $\Gamma$, we require some assumptions on the spectrum of the linearization about $\Gamma$. Because the manifold $\Gamma$ is invariant under the flow, we cannot assume that the linearized dynamics is contracting in directions tangential to the manifold. We must thus employ a more general notion of stability \cite{BLZ00}. The following definition will apply to solutions to partial differential equations that are invariant under a continuous group action (see the many examples in \cite{KP13,Schneider2017}), and also to rotating waves \cite{KMZ22}. To this end, for $\alpha \in \mathcal{S}$, let $\mathcal{L}_{\alpha} \in L\big( D(A) ,H \big)$ be the linearization about $\gamma_{\alpha}$, i.e.
\begin{align}
\mathcal{L}_{\alpha} = A + DN(\gamma_{\alpha}).
\end{align}
In order that we may correctly linearize the dynamics, we first require an expression for the flow restricted to the manifold $\mathcal{S}$ (In more detail: if the dynamics is periodic, for instance, then the linearization has to be with respect to coordinates that co-rotate with the limit cycle). 

Let the flow map of \eqref{eq:ODE}~be $(t,x)\,\mapsto\,\phi_t(x)$, and note that the flow can be expressed using Duhamel's formula as 
\begin{align}
\phi_t(x)\,=\, \Lambda_t x + \int_0^t \Lambda_{t-s}N(\phi_s(x))\,ds.  \label{eq: phi t x evolution operator}
\end{align}
Since the manifold $\Gamma$ is invariant under the flow, for all $t\geq 0$ there must exist an evolution operator $\tilde{\phi}_t: \mathcal{S} \to \mathcal{S}$ such that 
\begin{align}
\phi_t(\gamma_{\alpha}) = \gamma_{\tilde{\phi}_t(\alpha)}. \label{eq: tilde phi t definition}
\end{align}
We write $\tilde{N}: \mathcal{S} \to \mathbb{R}^m $ to be such that (defining $\tilde{\phi}_t(\alpha(0)) = \alpha(t)$),
\begin{align}
\frac{d\alpha^i}{dt} = \tilde{N}^i(\alpha(t)), \label{eq: tilde N definition}
\end{align}
and note that $\tilde{N}$ is continuous. Next we require an expression for the linearized dynamics about $\Gamma$. 
For some $\alpha \in \mathcal{S}$, write $U_{\alpha}(t) \in L(H,H)$ to be the (inhomogeneous) semigroup generated by $\mathcal{L}_{\tilde{\phi}_t(\alpha)}$. That is, for any $z$ in the domain of $A$, define $U_{\alpha}(t)z := x(t)$ where $x(t)$ is such that $x(0) = z$ and
\begin{align}
\frac{dx}{dt} =\mathcal{L}_{\tilde{\phi}_t(\alpha)}x(t). 
\end{align}
It is assumed that this definition can be continuously extended to define a semigroup in $L(H,H)$.

We are going to assume that the semigroup is contractive in all directions that are `normal' to the manifold $\Gamma$. To specify this more precisely, we define the projection operator $P_{\alpha}(t): H \mapsto H$ (for $t\geq 0$) to correspond to the projection of $U_{\alpha}(t)$ onto the non-decaying directions. Suppose that $\lbrace \psi_{\alpha}^{*,i} \rbrace_{1\leq i \leq m-1} \subset H$ is a set of linearly independent vectors in the kernel of $\mathcal{L}^*_{\alpha}$, and we define 
\begin{align}
P_{\alpha}(t) u &= \sum_{i=1}^n \langle \psi^{*,i}_{\alpha} , u \rangle \psi^{i}_{\tilde{\phi}_t(\alpha)}.
\end{align}
If $\Gamma$ were to consist entirely of fixed points, then the vectors could be chosen such that $P_{\alpha}(t)$ is the spectral projection operator (see \cite[4.1.7]{KP13}). Conversely if $\Gamma$ were the orbit of a stable limit cycle, then the vectors could be chosen such that $P_{\alpha}(t)$ is the Floquet projection matrix associated with the zero Floquet exponent. We require that the dynamics decays exponentially (and at a rate that is uniform over $\mathcal{S}$) towards $\Gamma$ in sufficiently small neighborhoods (except in directions parallel to the kernel of $\mathcal{L}_{\alpha}$). This is encapsulated by the following assumption.
\begin{enumerate}
\item[{\bf D.}] 
Define $V_{\alpha}(t) =  U_{\alpha}(t) - P_{\alpha}(t) $. We assume that there exist constants $C,b > 0$ such that for any $t \geq 0$, and any $u\in E$,
\begin{align} \label{eq: assumption exponential decay}
\sup_{\alpha\in \mathcal{S}}\norm{V_{\alpha}(t)}_E \leq & C \exp(-bt) \text{ and }\\
V_{\alpha}(t) \psi^{j}_{\alpha} =& 0 \text{ for all }1\leq j \leq m-1 \label{eq: orthonormal 0} \\
 \big\langle \psi^{*,j}_{\alpha} ,\psi^i_{\alpha} \big\rangle =& \delta(i,j) \label{eq: orthonormal} \\
\sup_{\alpha\in\mathcal{S}}\sup_{j\in \mathbb{N}}\big\lbrace \big| \big\langle \psi^{*,j}_{\alpha} , u \big\rangle \big| +  \big| \big\langle \psi^{j}_{\alpha} , u \big\rangle \big|  \big\rbrace \leq & C\norm{u}_E\label{eq: bound psi inner product} 
\end{align}
Furthermore it is assumed that there exists $\lbrace \lambda_{i} \rbrace_{i\geq 1} \subset \mathbb{R}^+$ such that for each $\alpha \in \mathcal{S}$, there exists an orthonormal basis (for $H$) $ \lbrace e_{\alpha,i} \rbrace_{i\geq 1} \subset E$, such that
\begin{align}
\langle e_{\alpha,i}, e_{\alpha,j} \rangle &= \delta(i,j) \label{eq: assumption basis 1} \\
\norm{V_{\alpha}(t) e_{\alpha,j}}_E &\leq C \exp\big(-t\lambda_{j} \big) \text{ for all }t\geq 0. \label{eq: assumption basis 3} 
\end{align}
We assume that $\alpha \mapsto e_{\alpha,i}$ is smooth, and that $\alpha \to \psi^{j}_{\alpha}$ and $\alpha \to \psi^{*,j}_{\alpha}$ are both three times continuously differentiable. (Note that, in practice, one would satisfy the constraints of this assumption by first obtaining bounds on the spectra of $L_{\alpha}$, and then applying the Gram-Schmidt algorithm to the corresponding eigenvectors).
\end{enumerate}


\begin{Def} \label{Definition Isochronal Phase Map}
It turns out that Assumptions {\bf C} and {\bf D}~imply that for each $x\in B(\Gamma)$ there is a unique $\pi(x)\in\mathcal{S}$ such that 
\begin{equation}
\label{eq:piDef}
\norm*{\phi_t(x)-\phi_t(\gamma_{\pi(x)})}_E \,\xrightarrow[t\rightarrow\infty]{}\,0, 
\end{equation} 
providing a well-defined map $\pi:B(\Gamma)\rightarrow\mathcal{S}$. 
See Theorem \ref{thm:C2}. 
We refer to $\pi$ as the \emph{isochron map}~of $\Gamma$. 
If \eqref{eq:SDE}~has a well-posed solution $(X_t)_{t\ge0}$ such that $X_t\in B(\Gamma)$ for some $t>0$, we refer to $\pi(X_t)$ as the \emph{isochronal phase}~of $X_t$. 
\end{Def}

Now, as noted in the introduction, we are concerned with stochastic perturbations of \eqref{eq:ODE}, and how they affect the stability and long-term behaviour of solutions in $B(\Gamma)$ in any of the cases listed above. 
Specifically, this paper concerns the stochastic evolution equation \eqref{eq:SDE}, 
written more concisely as 
\begin{equation}
dX_t\,=\,V(X_t)\,dt + \sigma B(X_t)\,dW_t. 
\end{equation}
The noise in \eqref{eq:SDE}~is defined to be a cylindrical Wiener process $W$ on $H$, so that it is uncorrelated in space and time. 
The case of spatially correlated noise can be treated by adding correlation via the diffusion coefficient $B$. 
We make the following assumptions on the noise, which holds for instance when $H=L^2(O;\R^n)$ for bounded $O$, and $B$ is polynomial or constant. 

\begin{enumerate}[{\bf A.}] 
\item [{\bf E.}] 
The multiplicative noise operator
\[
B: E \to \mathcal{L}( H, H) 
\]
is assumed to have the following properties. For a constant $C_B > 0$, and a locally square integrable mapping
$ t\to K(t)$, for all $s,t \geq 0$,
\begin{align*}
\norm{B(x) - B(y)}_E &\leq C_B \norm{x-y}_E \text{ for all }t\geq 0 \text{ and }x,y \in E \\
\norm{B(x)}_E &\leq C_B\\
\norm{P^A_t B(x)}_{HS} &\leq K(t)\\
\norm{P^A_t B(x) - P^A_t B(y)}_{HS} &\leq K(t)\norm{x-y}
\end{align*}
It is also assumed that for any orthonormal basis $\lbrace e_k \rbrace_{k\geq 1}$ of $H$,
\begin{align} \label{eq: bounded sum}
 \sup_{x\in \Gamma_{\delta}} \lim_{p\to\infty}\sum_{j \geq p}\sum_{k,l=1}^{\infty}(\lambda_k + \lambda_l)^{-1} \langle e_{\pi(x),k} , B(x) e_{j} \rangle \langle e_{\pi(x),l} , B(x) e_{j} \rangle  =0 .
\end{align}
\end{enumerate}
We note that in reaction diffusion systems on a compact one-dimensional domain, $\lambda_l \simeq l^2$, and therefore the boundedness of $B(x)$ would ensure that \eqref{eq: bounded sum} is satisfied. For reaction-diffusion systems on domains of higher dimension, further constraints on $B$ would be required in order that \eqref{eq: bounded sum} is satisfied.

\begin{lemma}
There exists a unique adapted stochastic process $(X_t)_{t\ge0}$ taking values in $C([0,\infty);E)$ such that 
\begin{equation}
\label{eq:Zmild}
X_t\,=\,\Lambda_tX_0+\int_0^t\Lambda_{t-s}N(X_s)\,ds + \sigma\int_0^t\Lambda_{t-s}B(X_s)\,dW_s. 
\end{equation}
The solution expressed by \eqref{eq:Zmild}~will henceforth be denoted $(X_t)_{t\ge0}$. 
\end{lemma}
\begin{proof}
Our assumptions dictate that
\[
\int_0^t\Lambda_{t-s}B(X_s)\,dW_s
\]
is a well-defined $E$-valued martingale. One can then prove the existence of solutions to \eqref{eq:Zmild} using a standard fixed-point argument \cite{DPZ14}. 

\end{proof}

\subsection{Applications}


One important class of applications is when $\Gamma$ is a manifold of fixed points. In the study of pattern formation in partial differential equations, usually the patterned solution is not unique, and one can obtain another solution through translating and / or rotating the original solution. As such, one obtains a continuous manifold of solutions to the original partial differential equation \cite{KP13,Schneider2017}. The previous work \cite{M22} of one of these authors provides an extensive list of examples of such systems (including details of how to add space-time white noise). Another very important class of examples is when $\Gamma$ consists of time-periodic solutions, such as rotating spiral waves. See \cite{nakao2014phase} and \cite{KMZ22} for recent analyses of the effect of noise on such systems.

\begin{rmk}
\label{rmk:EX} 
Two classes of examples of \eqref{eq:ODE}~motivate the present discussion. 
The first is the class of reaction-diffusion equations of the form 
\begin{equation}
\label{eq:RDE}
\partial_tx\,=\,(\Delta) x  + N(x) + \xi_t, 
\end{equation}
where $a>0$. 

%

The second class of examples in which we are interested in is a class of integro-differential equations, with dynamics at $\xi\in O\subseteq \mathbb{R}^d$ described as 
\begin{equation}
\label{eq:NFE}
\begin{aligned}
\partial_tx(t,\xi)\,&=\, -x(t,\xi) +  \int_O \omega(\xi,\zeta)f(x(t,\zeta))\,d\zeta ,  \\
\partial_ty(t,\xi)\,&=\, -\varepsilon^{-1}y(t,\xi) + x(t,\xi). 
\end{aligned}
\end{equation}
where $\omega$ and $f$ are typically bounded and Lipschitz. These equations are frequently employed in neuroscience (where they are typically referred to as \emph{neural field equations})~\cite{B12}~and ecology \cite{patterson2020probabilistic}. 
\end{rmk}

\section{The Isochron Map \&~an It{\^o}~Formula}
\label{sec:Regularity}
In this section we prove the existence the isochron map, obtain some regularity properties, before using these to obtain an It{\^o}~formula for the isochronal phase $\pi(X_t)$ of the SPDE \eqref{eq:SDE}. The first half of this section defines the isochronal phase map in a neighbourhood about the manifold $\Gamma$ using a fixed-point argument. We then demonstrate that the fixed-point identity can be implicitly differentiated, allowing us to prove that the isochronal phase map is twice continuously differentiable. We can then prove that the second derivative of the isochronal phase map is trace class, which in turn allows us to use these results to determine an SDE for the dynamics of the isochronal co-ordinate. 

We start by precisely defining the isochronal phase map using a fixed point argument.
\begin{thm}
\label{thm:C2} 
The function $\pi:B(\Gamma)\rightarrow\mathcal{S}$ (as written in \eqref{eq:piDef}) is well-defined.
\end{thm}
Throughout this paper, we let $\Gamma_\delta$ denote the following $\delta$-neighbourhood of $\Gamma$, 
\[
\Gamma_\delta\,\coloneqq\,\left\{x\in E\,:\,\norm*{x-\gamma_{\pi(x)}}_E\le\delta\right\}  
\]
for $\delta>0$. 
Our assumptions dictate that the map $x\mapsto\phi_t(x)$ is three times Fr{\'e}chet differentiable in the topology of $E$ at any $x\in\Gamma_\delta$. 
The first, second, and third derivatives of the evolution operator $\phi_t$ at $x_0$ (as defined in \eqref{eq: phi t x evolution operator}) are denoted 
\begin{equation}
\label{eq:Derivatives}
x\,\mapsto\,D\phi_t(x_0)[x],\qquad (x,y)\,\mapsto\,D^2\phi_t(x_0)[x,y],\qquad (x,y,z)\,\mapsto\,D^3\phi_t(x_0)[x,y,z], 
\end{equation}
in directions $x,y,z\in E$.  We briefly note that these derivatives are well-defined.
\begin{lemma} \label{Frechet Derivative Evolution Operator}
For any $x_0 \in E$ and $y,z,w\in E$, there exists a unique $y \in \mathcal{C}([0,\infty), E)$ such that
\begin{align} \label{eq: y definition 0 0 }
y(t) = \Lambda_t y + \int_0^t \Lambda_{t-s} DN\big( \phi_s(x_0) \big)[y(s)] ds \\
\end{align}
Furthermore, $D\phi_t(x_0)[y] = y(t)$. Next, writing $x_s =  \phi_s(x_0)$, there exists a unique $a \in  \mathcal{C}([0,\infty), E)$ such that
\begin{align}
a(t) =&\int_0^t \Lambda_{t-s}\big\lbrace DN( x_s )[a(s)] + D^2N(x_s )\big[ D\phi_s(x)[y] ,  D\phi_s(x)[z]  \big] \big\rbrace ds,
\end{align}
and $a(0) = 0$. Furthermore,  $D^2\phi_t(x_0)[y,z] = a(t)$. Finally, there exists a unique $b \in \mathcal{C}([0,\infty), E)$  such that for all $t\geq 0$,
\begin{multline}
b(t) =\int_0^t \Lambda_{t-s}\bigg\lbrace DN( x_s )[b(s)] + D^2N(x_s )\big[ D^2\phi_s(x)[y,w] ,  D\phi_s(x)[z]  \big] + \\
+ D^2N(x_s )\big[ D\phi_s(x)[y] ,  D\phi_s(x)[w,z]  \big] + D^2N(x_s )\big[ D\phi_s(x)[y,z] ,  D\phi_s(x)[w]  \big] \\+ D^3N(x_s )\big[ D\phi_s(x)[y] ,  D\phi_s(x)[z], D\phi_s(x)[w]   \big] \bigg\rbrace ds,
\end{multline}
and $b(0) = 0$. Furthermore,  $D^3\phi_t(x_0)[y,z,w] = b(t)$.
\end{lemma}
\begin{proof}
For small enough $t$, the existence and uniqueness of \eqref{eq: y definition 0 0 } follows from a fixed point argument. One then considers the equation, for $r \geq t$,
\begin{align} \label{eq: y definition 0 0 0 }
y(r) = \Lambda_{r-t} y(t) + \int_t^r \Lambda_{r-s} DN\big( \phi_s(x_0) \big)[y(s)] ds.
\end{align}
Again, for $r-t$ small enough, a fixed point argument implies the existence of a unique $y(r)$. Continuing this argument, one finds a unique $y(r)$ satisfying \eqref{eq: y definition 0 0 0 } for all $r \geq t$. One can then see that, implicitly differentiating the expression in \eqref{eq: phi t x evolution operator},
$D\phi_t(x_0)[y] = y(t)$. The proofs of the other identities are analogous.
\end{proof}

It suffices to prove that there exists $\delta > 0$ such that the map $\pi$ is well-defined for any $x\in \Gamma_{\delta}$  (the closed $\delta$-blowup of $\Gamma$), since eventually the deterministic dynamics must always enter $\Gamma_{\delta}$. We prove the existence of the map $\pi$ by adapting the fixed point argument of \cite[Chapter 5]{volpert1994traveling}. To this end, define $\mathcal{V} \subset \mathcal{C}([0,\infty), E)$ to be the Banach Space of all functions such that
\begin{align}
\mathcal{V} &= \big\lbrace u \in  \mathcal{C}([0,\infty), E) \; : \; \sup_{t\geq 0}\norm{ u(t)}_E \exp(bt) < \infty \big\rbrace \text{ with norm } \\
\norm{u}_b &= \mu\sup_{t\geq 0} \norm{u(t) }_{E} \exp(bt) ,
\end{align}
for a constant $\mu > 0$, and $b > 0$ bounds the rate of decay of the linearization (as noted in Assumption D). Define the function $\Lambda: \mathcal{S} \times \mathcal{V} \to \mathcal{V}$ to be such that
\begin{align}
\Lambda(\alpha,v)(t) =& V_{\alpha}(t)(x - \gamma_{\alpha}) + \int_0^t V_{\alpha}(t-s)G(s,\alpha , v(s)) ds\nonumber \\ &- \sum_{j=1}^{m-1} \psi^j_{\tilde{\phi}_t(\alpha)} \int_t^{\infty} \big\langle  \psi^{*,j}_{\tilde{\phi}_s(\alpha)} , G(s,\alpha , v(s))   \big\rangle ds \text{ where }  \\
G(s,\alpha,v) =& N( v + \gamma_{\tilde{\phi}_s(\alpha)} ) - N(\gamma_{\tilde{\phi}_s(\alpha)}) - DN(\gamma_{\tilde{\phi}_s(\alpha)}) [ v]. 
\end{align}
For $1\leq j \leq m-1$ and for some $\kappa \in \mathcal{S}$, define the function $\Pi^j_{\kappa}: \mathcal{S} \times \mathcal{V} \to \mathbb{R}$ to be such that 
\begin{multline} \label{eq: H kappa definition}
\Pi_{\kappa}^j(\alpha,v) = \big\langle x- \gamma_{\kappa} , \psi^{*,j}_\alpha \big\rangle - \big\langle \gamma_{\alpha} - \sum_{i=1}^m (\kappa^i - \alpha^i) \psi^i_{\alpha}- \gamma_{\kappa} , \psi^{*,j}_\alpha  \big\rangle \\ +  \int_0^{\infty} \big\langle  \psi^{*,j}_{\tilde{\phi}_s(\alpha)} , G(s,\alpha , v(s))   \big\rangle ds . 
\end{multline}
It is assumed throughout this section that $\norm{x - \gamma_{\kappa}}_E \leq \delta$ for some $\delta > 0$.

\begin{thm} \label{Theorem existence isochronal}
As long as $\delta , \bar{\alpha} $ are sufficiently small, there exists a unique $\alpha \in \mathcal{S}$ with $\norm{\alpha-\kappa}\leq \bar{\alpha}$ and $v \in \mathcal{V}$ such that
\begin{align}
\Lambda(\alpha,v) =& v  \\
\Pi^j_{\kappa}(\alpha,v) =& \alpha^j - \kappa^j \text{ for all }1\leq j \leq m ,
\end{align}
and we define $\xi(x,\alpha,s) = v(s)$, and $\pi(x) = \alpha$. (The function $\pi: \Gamma_{\delta} \to \mathcal{S}$ is the isochronal phase map noted earlier in Definition \ref{Definition Isochronal Phase Map}). Furthermore, writing $u(t) =\xi(x,\alpha,t)  + \gamma_{\tilde{\phi}_t(\alpha)}$, it must be that
\begin{align}
u(t) = U_{\alpha}(t)u(0) + \int_0^t U_{\alpha}(t-s)G(s,\alpha , u(s) - \gamma_{\tilde{\phi}_s(\alpha)}) ds , \label{eq: to show u 1}
\end{align}
which is equivalent to
\begin{equation}
u(t) = \Lambda_t u(0) + \int_0^t \Lambda_{t-s}N(u(s)) ds. \label{eq: to show u 2}
\end{equation}
\end{thm}
\begin{proof}
The proof is a very straightforward adaptation of \cite[Theorem 1.1]{volpert1994traveling}. Writing $\mu = \delta^{1/4}$, $r= \mu^3$, consider the closed subset $\tilde{\mathcal{V}}_{\kappa,r,\bar{\alpha}}$ of the Banach Space $\mathcal{S} \times \mathcal{V}$ such that
\begin{align}
\tilde{\mathcal{V}}_{\kappa,r,\bar{\alpha}} = \big\lbrace (\alpha,v) \in \mathcal{S} \times \mathcal{V} : \norm{v}_{b} \leq r \; \; , \; \; \norm{\alpha - \kappa} \leq \bar{\alpha}  \big\rbrace .
\end{align}
Next we claim that there exists a constant $c > 0$ (independent of the other constants) such that for any $(\alpha,v) , (\beta,w) \in \tilde{\mathcal{V}}_{\kappa,r}$, 
\begin{align}
\norm{\Pi_{\kappa}(\alpha,v)} &\leq c \big( \delta + \bar{\alpha}^2 + r^{2}\mu^{-2} \big) \\
\norm{\Lambda(\alpha,v)}_b &\leq c \mu  \big( \delta + \bar{\alpha}^2 + r^{2}\mu^{-2} \big) \\
\norm{\Pi_{\kappa}(\alpha,v) - \Pi_{\kappa}(\beta,w)} &\leq c \big( \delta + \bar{\alpha}^2 + \bar{\alpha} + r\mu^{-1}  + r^2 \mu^{-2} \big)\norm{\alpha-\beta} + c r\mu^{-2}\norm{v-w}_{b} \\
\norm{ \Lambda(\alpha,v) -  \Lambda(\beta,w)}_b &\leq c \big( \delta \mu+ \bar{\alpha} \mu + \mu+ r^2 \mu^{-1}  + r  \big)\norm{\alpha-\beta} + c r\mu^{-1}\norm{v-w}_{b}.
\end{align}
Observe that, for instance, writing $K > 0$ to be a uniform upperbound for $D^2 N$, for $0\leq t \leq \infty$,
\begin{align}
\left\| \int_0^t G(s,\alpha,v) ds \right\|_E &\leq \int_0^t K \norm{v(s)}^2_E ds \nonumber \\ &\leq  \sup_{s\geq 0} \big\lbrace \exp(2bs) \norm{v(s)}^2_E \big\rbrace \int_0^t K \exp(-2bs)ds 
\leq \mu^{-2} \big(\norm{v}_b \big)^2 \frac{K}{2b} \leq \frac{\delta K}{2b } \label{eq: bound integral G}
\end{align}
Observe that if one chooses the constants such that, then if $(\alpha,v) \in \tilde{\mathcal{V}}_{\kappa,r,\bar{\alpha}} $, then necessarily
\begin{align}
\big( H_{\kappa}(\alpha,v),\Lambda(\alpha,v) \big) \in \tilde{\mathcal{V}}_{\kappa,r,\bar{\alpha}}.
\end{align}
Furthermore for a constant $\zeta < 1$, 
\begin{align}
\norm{\Pi_{\kappa}(\alpha,v) - \Pi_{\kappa}(\beta,w)} + \norm{ \Lambda(\alpha,v) -  \Lambda(\beta,w)}_b \leq \zeta \big\lbrace \norm{\alpha-\beta} + \norm{v-w}_b \big\rbrace .
\end{align}
It thus follows from the contraction mapping principle that there exists a unique $(\alpha,v) \in \tilde{\mathcal{V}}_{\kappa,r,\bar{\alpha}}$ such that
\begin{align}
\alpha^i - \kappa^i= \Pi^i_{\kappa}(\alpha,v) \text{ and } v= \Lambda(\alpha,v) .
\end{align}
Next, rearranging terms, and using the identity in \eqref{eq: orthonormal 0}, we find that it must be that
\begin{align}
\Xi^j( \pi(x) , x) =& 0 \text{ for all }1\leq j \leq m-1, \text{ where }\\
\Xi^j(\alpha,x) :=&  \big\langle  \psi^{*,j}_{ \alpha} ,  x - \gamma_{\alpha} \big\rangle + \int_0^{\infty}\big\langle  \psi^{*,j}_{\tilde{\phi}_s(\alpha)}, G(s,\alpha , \xi(x,\alpha,s) )  \big\rangle ds  \\
 =&   \big\langle  \psi^{*,j}_{ \alpha} ,  x - \gamma_{\alpha} \big\rangle + \int_0^{\infty}\big\langle  \psi^{*,j}_{\tilde{\phi}_s(\alpha)}, \tilde{G}(s,\alpha , x )  \big\rangle ds \text{ where } \label{eq: Xi j definition}\\
 \tilde{G}(s,\alpha,x ) =& N\big( \phi_s(x) \big)- DN\big(\gamma_{\tilde{\phi}_s(\alpha)}\big)\big[ \phi_s(x) - \gamma_{\tilde{\phi}_s(\alpha)} \big] -  N\big(\gamma_{\tilde{\phi}_s(\alpha)}\big)
\end{align}
These identities imply that $u(t) := \gamma_{\tilde{\phi}_t(\alpha)} + v(t)$ must satisfy \eqref{eq: to show u 1} and \eqref{eq: to show u 2}.
\end{proof}
The following corollary is immediate.
\begin{cor}\label{Cor isochronal phase invariant}
For all $x\in \Gamma_{\delta}$ and all $t\geq 0$,
\begin{equation}
\pi\big( \phi_t(x) \big) = \tilde{\phi}_t(x),
\end{equation}
where $\tilde{\phi}_t$ is the flow operator restricted to the manifold $\Gamma$ (as defined in \eqref{eq: tilde phi t definition}).
\end{cor}
Our next goal is to determine expressions for the first and second Fr{\'e}chet Derivatives of the function $\pi$. These are essential in order that we may obtain an SDE for $\pi(X_t)$. First, we prove that the linearization can be made with respect to the isochronal phase, without a significant loss of accuracy.
\begin{lemma} \label{Lemma Frechet Derivative Flow Map}
There exists $\tilde{\delta} \leq \delta$ ($\delta$ is stipulated in Theorem \ref{Theorem existence isochronal}) such that for any $x\in \Gamma_{\tilde{\delta}}$, if $y(t)$ solves the following linear equation (for any $y\in E$ and any $t\geq 0$),
\begin{align} \label{eq: y definition}
y(t) = \Lambda_t y + \int_0^t \Lambda_{t-s} DN\big( \phi_s(x) \big)[y(s)] ds,
\end{align}
then necessarily for all $t \geq 0$,
\begin{align}
\norm{y(t)}_E \leq 2 \norm{U_{\pi(x)}(t) y}_E
\end{align}
Furthermore
\begin{equation}
D\phi_t(x)[y] = y(t).
\end{equation}
\end{lemma}
\begin{proof}
One easily proves existence and uniqueness of the solution to \eqref{eq: y definition} by successively employing a fixed point argument over sufficiently small time intervals. We can equivalently write the solution of \eqref{eq: y definition}  as \cite{ENB00}
\begin{align}\label{eq: equivalent semigruop}
y(t) =U_{\pi(x)}(t) y(0) + \int_0^t U_{\pi(x)}(t-s) \big\lbrace DN\big( \phi_s(x) \big)[y(s)] - DN\big( \gamma_{\tilde{\phi}_s(\pi(x))}  \big)[y(s)]  \big\rbrace ds.
\end{align}
Thanks to the Taylor Remainder Theorem, there exists $\kappa_s$ in the convex hull of $\phi_s(x)$ and $ \gamma_{\tilde{\phi}_s(\pi(x))}$ such that
\[
DN\big( \phi_s(x) \big)[y(s)] - DN\big( \gamma_{\tilde{\phi}_s(\pi(x))}  \big)[y(s)] = D^2N\big( \kappa_s \big)\big[ y(s) ,  \phi_s(x) - \gamma_{\tilde{\phi}_s(\pi(x))}   \big].
\]
Since, by assumption, the second Fr{\'e}chet  Derivative of $N$ is uniformly upperbounded (in the operator topology of $E$) over $\Gamma_{\delta}$, there exists a constant such that 
\begin{align} \label{eq: taylor second temporary}
\norm{ D^2N\big( \gamma_{\tilde{\phi}_s(\pi(x))}  \big)\big[ y(s) ,  \phi_s(x) - \gamma_{\tilde{\phi}_s(\pi(x))}   \big]}_E \leq \rm{Const} \norm{y(s)}_E \norm{ \phi_s(x) - \gamma_{\tilde{\phi}_s(\pi(x))}}_E
\end{align}
Furthermore, thanks to Theorem \ref{Theorem existence isochronal}, through taking $\delta$ to be sufficiently small, we can ensure that for arbitrary $\tilde{\delta} > 0$, for all $s \geq 0$,
\begin{equation}\label{eq: difference bound}
\norm{ \phi_s(x) - \gamma_{\tilde{\phi}_s(\pi(x))}}_E \leq \tilde{\delta} \exp(-bs).
\end{equation}
Let 
\begin{align}
\hat{\tau} = \inf\big\lbrace t\geq 0 : \norm{y(t)}_E = 2 \norm{U_{\pi(x)}(t)y }_E \big\rbrace.
\end{align}
For all $t \leq \hat{\tau}$, it follows from \eqref{eq: equivalent semigruop}, \eqref{eq: taylor second temporary} and \eqref{eq: difference bound} and the fact that $\norm{U_{\alpha}(t)} \leq \rm{Const}_2$ (this is the operator norm with respect to the norm on $E$), that
\begin{align}
\norm{y(t)}_E \leq \norm{U_{\pi(x)}(t) y }_E + \rm{Const}\times \rm{Const}_2 \int_0^t \tilde{\delta}\exp(-bs)2 \norm{U_{\pi(x)}(s) y }_E  ds.
\end{align}
As long as $\tilde{\delta}$ is sufficiently small that $2b^{-1}\rm{Const}\times \rm{Const}_2 \tilde{\delta} < 1$, we find that $\hat{\tau} = \infty$ and for all $t > 0$,
\[
\norm{y(t)}_E < 2  \norm{U_{\pi(x)}(t) y }_E 
\]
\end{proof}

\begin{lemma} \label{Lemma Second Frechet Derivative Flow Map}
There exists $\tilde{\delta} \leq \delta$ ($\delta$ is stipulated in Theorem \ref{Theorem existence isochronal}) such that for any $x\in \Gamma_{\tilde{\delta}}$ and any $t\geq 0$, $\phi_t(x)$ is second Fr{\'e}chet -Differentiable. For any $z,y \in E$, the second Fr{\'e}chet  Derivative is written as $D^2\phi_t(x)[y,z]$, and is such that, writing $x_s = \phi_s(x)$,
\begin{align}
D^2\phi_t(x)[y,z] =& u(t) \text{ where } \\
u(t) =& \int_0^t \Lambda_{t-s}\big\lbrace DN(x_s)[u(s)] + D^2N(x_s)\big[ D\phi_s(x)[y] ,  D\phi_s(x)[z]  \big] \big\rbrace ds.
\end{align}
Furthermore there is a constant such that for all $x\in \Gamma_{\tilde{\delta}}$, and all $y,z \in E$,
\begin{align}
\norm{D^2\phi_t(x)[y,z]}_E \leq \rm{Const} \int_0^t \norm{D\phi_s(x)[y]}_E \norm{D\phi_s(x)[z]}_E ds
\end{align}
\end{lemma}
\begin{proof}
By assumption, $N$ is twice continuously differentiable. This proof is thus analogous to that of Lemma \ref{Lemma Frechet Derivative Flow Map} and is neglected.
\end{proof}

Our next step is to prove that $\pi$ is Fr{\'e}chet-differentiable at some $x\in \Gamma_{\delta}$. To do this, we will apply the implicit function theorem to the map $\Xi^j$ defined in \eqref{eq: Xi j definition}. First, let $D\Xi^j(\alpha,x)[y]$ denote the Fr{\'e}chet  Derivative of $\Xi^j$ with respect to $x$ only (in the direction $y\in E$), i.e. (formally)
\begin{multline}
D\Xi^j(\alpha,x)[y] := \lim_{\epsilon \to 0^+} \epsilon^{-1}\big\lbrace \Xi^j(\alpha,x + \epsilon y) -\Xi^j(\alpha,x ) \big\rbrace  \\ = \big\langle  \psi^{*,j}_{ \alpha} , y \big\rangle  + \int_0^{\infty}\big\langle  \psi^{*,j}_{\tilde{\phi}_s(\alpha)},  DN(\phi_s(x))\big[ D\phi_s(x)[y] \big] -DN\big( \gamma_{\tilde{\phi}_s(\alpha)}\big)\big[ D\phi_s(x)[y] \big]  \big\rangle ds . \label{eq: first derivative Xi}
\end{multline}

\begin{lemma} \label{Lemma Existence Limit Derivative}
The limit in \eqref{eq: first derivative Xi} exists for any $x\in \Gamma_{\delta}$, and defines a bounded linear operator $D\Xi^j(\alpha,x): E \to \mathbb{R}$. 
\end{lemma}
\begin{proof}
We check that the above expression is indeed well-defined. First, thanks to Lemma \ref{Lemma Frechet Derivative Flow Map}, the Fr{\'e}chet  Derivative $D\phi_t(x)$ exists for all $t \geq 0$. Furthermore, thanks to Taylor's Theorem, for some $\kappa_s$ in the convex hull of $\phi_s(x)$ and $ \gamma_{\tilde{\phi}_s(\alpha)}$
\begin{align*}
 DN(\phi_s(x))\big[ D\phi_s(x)[y] \big] -DN\big( \gamma_{\tilde{\phi}_s(x)}\big)\big[ D\phi_s(x)[y] \big]  = D^2N\big( \kappa_s \big) \big[ D\phi_s(x)[y] , \phi_s(x) - \gamma_{\tilde{\phi}_s(\alpha)} \big] .
\end{align*}
Since the second Fr{\'e}chet  Derivative of $N$ is uniformly upperbounded (by assumption), we find that
\begin{align*}
\norm{ DN(\phi_s(x))\big[ D\phi_s(x)[y] \big] -DN\big( \gamma_{\tilde{\phi}_s(x)}\big)\big[ D\phi_s(x)[y] \big] }_E \leq \rm{Const} \norm{ D\phi_s(x)[y]}_E \norm{\phi_s(x) - \gamma_{\tilde{\phi}_s(\alpha)}}_E.
\end{align*}
The existence proof in Theorem \ref{Theorem existence isochronal} implies that $\norm{ \phi_s(x) - \gamma_{\tilde{\phi}_s(\pi(x))}}_E \leq \rm{Const}\exp(-bs)$. This allows us to conclude that \eqref{eq: first derivative Xi} is well-defined.
\end{proof}
Before we can use the implicit function theorem on $\Xi$ to determine the derivative of $\pi$, we must determine an expression for the derivative of $\Xi$ with respect to $\alpha$. Write $\tilde{E}_{\delta} \subset E \times \mathcal{S}$ to be the set of all $(x,\alpha) \in \Gamma_{\delta} \times \mathcal{S}$ such that
\begin{align}
\norm{x - \gamma_{\alpha}}_E \leq \delta .
\end{align}
Define $M: \tilde{E}_{\delta} \to \mathbb{R}^{(m-1)\times (m-1)}$ as follows. Suppose that $( \alpha_{\epsilon})_{\epsilon \geq 0} \subset \mathcal{S}$ is such that $\lim_{\epsilon \to 0^+}\epsilon^{-1}\lbrace \alpha_{\epsilon} - \alpha \rbrace = \mathfrak{v}_{\alpha,k}$, and define
\begin{align}
M^{jk}(x,\alpha) = \lim_{\epsilon \to 0^+} \epsilon^{-1} \lbrace \Xi^j(\alpha_{\epsilon},x) -  \Xi^j(\alpha,x)\rbrace. \label{eq: M jk definition}
\end{align}
In other words, the derivative is in the direction $\mathfrak{v}_{\alpha,k}$ in the tangent plane of $\mathcal{S}$ at $\alpha$. Our assumptions on the smoothness of the manifold $\Gamma$ imply that the limit in \eqref{eq: M jk definition} is well-defined, regardless of the choice of the sequence $\alpha_{\epsilon}$.


\begin{lemma} \label{Lemma Invertible Matrix}
As long as $\delta$ is sufficiently small, $M(x,\alpha)$ is invertible for all $(x,\alpha) \in \tilde{H}_{\delta}$. Write the matrix-inverse as $O(x,\alpha) \in \mathbb{R}^{(m-1) \times (m-1)}$, with elements $\big( O_{ij}(x,\alpha) \big)_{1\leq i,j \leq m-1}$.
\end{lemma}
\begin{proof}
 We let $\partial_k$ denote the directional derivative. In more detail, suppose that $( \alpha_{\epsilon})_{\epsilon \geq 0} \in \mathcal{S}$ is such that $\lim_{\epsilon \to 0^+}\epsilon^{-1}\lbrace \alpha_{\epsilon} - \alpha \rbrace = \mathfrak{v}_{\alpha,k}$, then we write
\begin{align}
\partial_k \psi^{*,j}_{\alpha} = \lim_{\epsilon\to 0^+}\epsilon^{-1}\lbrace \psi^{*,j}_{\alpha_{\epsilon}} -\psi^{*,j}_{\alpha}  \rbrace \label{eq: psi j alpha directional derivative}
\end{align}
We compute that
\begin{multline}
\lim_{\epsilon \to 0^+}\epsilon^{-1} \lbrace \Xi^j(\alpha_{\epsilon},x) - \Xi^j(\alpha,x) \rbrace= \big\langle \partial_{k} \psi^{*,j}_{\alpha} , x - \gamma_{\alpha} \big\rangle - \big\langle \psi^{*,j}_{\alpha}, \psi^k_{\alpha} \big\rangle 
 +  \int_0^{\infty} \big\langle \partial_{k} \psi^{*,j}_{\tilde{\phi}_s(\alpha)}, \tilde{G}(s,\alpha , x )  \big\rangle ds \\
  -  \int_0^{\infty} \big\langle \psi^{*,j}_{\tilde{\phi}_s(\alpha)}, D^2N\big( \gamma_{\tilde{\phi}_s(\alpha)}\big)\big[\phi_s(x) - \gamma_{\tilde{\phi}_s(\alpha)} , \partial_{k}\gamma_{\tilde{\phi}_s(\alpha)} \big]  \big\rangle ds 
\end{multline}
Our assumption in \eqref{eq: orthonormal} dictates that 
\begin{equation}
 \big\langle \psi^{*,j}_{\alpha}, \psi^k_{\alpha} \big\rangle = \delta(j,k).\label{eq: infimum of determinant}
\end{equation}
%
Furthermore, it follows from our assumption in \eqref{eq: bound psi inner product}  that
\begin{align}
\big| \big\langle \partial_{k} \psi^{*,j}_{\alpha} , x - \gamma_{\alpha} \big\rangle \big| &\leq \rm{Const}\norm{x-\gamma_{\alpha}}  \\
\big|  \int_0^{\infty} \big\langle \partial_{k} \psi^{*,j}_{\tilde{\phi}_s(\alpha)}, \tilde{G}(s,\alpha , x )  \big\rangle ds \big| &\leq \rm{Const} \int_0^{\infty} \norm{ \tilde{G}(s,\alpha ,x ) }_E ds
\end{align}
Our assumptions dictate that $ \norm{ \partial_{k} \psi^{*,j}_{\tilde{\phi}_s(\alpha)}}$ (note that this is the $H$-norm) is uniformly upperbounded. The bound in \eqref{eq: bound integral G} implies that
\begin{align}
 \int_0^{\infty} \norm{ \tilde{G}(s,\alpha ,x ) }_E ds\leq  \delta K / (2b), \label{eq: norm integral G square}
\end{align}
which becomes arbitrarily small as $\delta \to 0$. It thus follows from \eqref{eq: infimum of determinant} and \eqref{eq: norm integral G square} that as long as $\delta$ and $r$ are sufficiently small, it must be that $M(x,\alpha)$ is invertible.
\end{proof}
The following corollary follows easily from the previous two lemmas (and its proof is neglected).
\begin{cor} \label{Corollary pi Frechet Differentiable}
$\pi$ is Fr{\'e}chet  Differentiable for all $x\in \Gamma_{\delta}$. The derivative of $\pi$ at $x \in \Gamma_{\delta}$ in the direction $y\in E$ is written as $D\pi(x)[y] \in \mathcal{T}_{\alpha}(\mathcal{S})$. Write $\big\lbrace D\pi_i(x)[y]  \big\rbrace_{i=1}^{m-1}$ to be the component in the direction of $\mathfrak{v}_{\alpha,i}$, that is 
\begin{align}
D\pi^k(x)[y] := \sum_{i=1}^{m-1} \mathfrak{v}_{\alpha,i}^k D\pi_i(x)[y]. \label{equation first frechet derivative}
\end{align}
These components assume the form
\begin{align}\label{eq: R definition}
D\pi_i(x)[y] = -\sum_{j=1}^{m-1}O_{ij}\big(x, \pi(x) \big)D\Xi^j\big( \pi(x),x \big)[y]  := \mathcal{R}_i(x,y).
\end{align}
Finally, there exists a constant such that (as long as $\delta$ is sufficiently small),
\begin{align}
\sup_{1\leq i \leq m-1}\sup_{x\in \Gamma_{\delta}, y\in E}\big| D\pi_i(x)[y]  \big| \leq \rm{Const}\norm{y}_E .
\end{align}
\end{cor}

Next, we must determine the second Fr{\'e}chet  Derivative of $\pi$. First, we note that the second Fr{\'e}chet  derivative of $\Xi^j(\alpha,x)$ (with respect to $x$ only) in the directions $y,z \in E$, can be formally written as 
\begin{multline} \label{eq: second derivative of Xi}
D^2\Xi^j(\alpha,x)[y,z] =   \int_0^{\infty}\bigg\langle  \psi^{*,j}_{\tilde{\phi}_s(\alpha)},  D^2N(\phi_s(x))\big[ D\phi_s(x)[y] , D\phi_s(x)[z] \big] \\+ DN(\phi_s(x))\big[ D^2\phi_s(x)[y,z] \big] -DN\big( \gamma_{\tilde{\phi}_s(\alpha)}\big)\big[ D^2\phi_s(x)[y,z] \big]  \bigg\rangle ds .
\end{multline}
The directional derivative (using the same notation as in \eqref{eq: psi j alpha directional derivative}) is
\begin{multline} \label{eq: second derivative of Xi again}
\partial_{k} D\Xi^j(\alpha,x)[y] = \big\langle \partial_{k} \psi^{*,j}_{ \alpha} , y \big\rangle  \\+ \int_0^{\infty}\big\langle \partial_{k} \psi^{*,j}_{\tilde{\phi}_s(\alpha)},  DN(\phi_s(x))\big[ D\phi_s(x)[y] \big] -DN\big( \gamma_{\tilde{\phi}_s(\alpha)}\big)\big[ D\phi_s(x)[y] \big]  \big\rangle ds \\
- \int_0^{\infty}\big\langle  \psi^{*,j}_{\tilde{\phi}_s(\alpha)}, \partial_{k} DN\big( \gamma_{\tilde{\phi}_s(\alpha)}\big)\big[ D\phi_s(x)[y] \big]  \big\rangle ds.
\end{multline}

\begin{lemma}
The derivatives in \eqref{eq: second derivative of Xi} and \eqref{eq: second derivative of Xi again} exist (defining (respectively) a bounded bilinear and bounded linear operator on $E$). $\pi$ is twice Fr{\'e}chet -differentiable in $\Gamma_{\delta}$ (as long as $\delta$ is small enough). For $y,z \in E$,
\begin{multline}
D^2 \pi_i(x)[y,z] = -\sum_{j=1}^{m-1}\big\lbrace DO^{ij}\big(x, \pi(x) \big)[z]D\Xi^j\big( \pi(x),x \big)[y] + O^{ij}\big(x, \pi(x) \big)D^2\Xi^j\big( \pi(x),x \big)[y,z]  \big\rbrace  \\
+ \sum_{j=1}^{m-1} \sum_{a=1}^m \mathcal{R}_j(x,z) \mathfrak{v}_{\pi,j}^a \partial_{a} \mathcal{R}_i(x,y) \label{eq: second pi derivative}
\end{multline}
\end{lemma}
\begin{proof}
The proof of the existence of the derivatives in \eqref{eq: second derivative of Xi} and \eqref{eq: second derivative of Xi again} is very similar to the proof of Lemma \ref{Lemma Existence Limit Derivative}. One can then directly differentiate the expression in \eqref{eq: R definition} to obtain \eqref{eq: second pi derivative}.
\end{proof}

In order that we may employ a modified  It{\^o} Lemma to obtain an SDE expression for $\pi(x_t)$, we require that the second derivative of the map $\pi$ is `trace class'. The following lemma obtains a bound on the second-derivative of the isochronal phase map with respect to the basis vectors $\lbrace e_{\pi(x),j} \rbrace_{j\geq 1}$ (these are defined in Assumptions D and E ).
\begin{lemma} \label{Lemma bound decay basis vectors}
Recall the orthonormal basis $\lbrace e_{\alpha,j} \rbrace_{j\geq 1} \subset E$ of $H$ outlined in \eqref{eq: assumption basis 1}-\eqref{eq: assumption basis 3}. There exists a constant such that for all $j,k \geq 1$,
\begin{align}
\sup_{x\in \Gamma_{\delta}, 1\leq i \leq m-1}\bigg| D^2 \pi_i(x)[e_{\pi(x),j},e_{\pi(x),k}]  \bigg| \leq \rm{Const} \times (\lambda_j + \lambda_k )^{-1}. 
\end{align}
Furthermore the following map is continuous
\[
x \mapsto D^2 \pi_i(x)[e_{\pi(x),j},e_{\pi(x),k}] .
\]
\end{lemma}
\begin{proof}
We start by bounding the terms in $D^2\Xi^l(\alpha,x)[ e_{\pi(x),j}, e_{\pi(x),k}] $ (as written in \eqref{eq: second derivative of Xi}). Thanks to Lemma \ref{Lemma Frechet Derivative Flow Map}, and the definition of the orthonormal basis in \eqref{eq: assumption basis 1}-\eqref{eq: assumption basis 3}, it must be that 
\begin{align}
\norm{D\phi_t(x)[e_{\pi(x),j}]}_E \leq C\exp(-\lambda_j t). \label{eq: decay linearization j again}
\end{align}
Thanks to the upperboundedness of the second Fr{\'e}chet  Derivative of $N$, and our assumption in \eqref{eq: bound psi inner product}, it must be that there is a constant such that for any $l$,
\begin{multline}
\bigg| \int_0^{\infty}\bigg\langle  \psi^{*,l}_{\tilde{\phi}_s(\alpha)},  D^2N(\phi_s(x))\big[ D\phi_s(x)[e_{\pi(x),j}] , D\phi_s(x)[e_{\pi(x),k}] \big] \bigg\rangle ds \bigg| \\ \leq \rm{Const} \times \int_0^\infty \exp\big(-s(\lambda_j + \lambda_k) \big)ds = \rm{Const}(\lambda_j + \lambda_k)^{-1}.
\end{multline}
We similarly obtain that there is a constant $K > 0$ such that
\begin{align*}
\bigg| \int_0^{\infty}\bigg\langle  \psi^{*,l}_{\tilde{\phi}_s(\alpha)}, &DN(\phi_s(x))\big[ D^2\phi_s(x)[e_{\pi(x),j}, e_{\pi(x),k}] \big] - DN\big( \gamma_{\tilde{\phi}_s(\alpha)}\big)\big[ D^2\phi_s(x)[e_{\pi(x),j}, e_{\pi(x),k}] \big]  \bigg\rangle ds \bigg| \\
\leq & \int_0^{\infty} K \exp(-bs) \int_0^s \exp\big(-r(\lambda_j + \lambda_k ) \big) dr ds \\
\leq & \rm{Const}(\lambda_j + \lambda_k)^{-1},
 \end{align*}
 using the bound in Lemma \ref{Lemma Second Frechet Derivative Flow Map}, and the bound $\norm{\phi_t(x) -  \gamma_{\tilde{\phi}_s(\alpha)}}_E \leq \rm{Const}\times \exp(-bt) $.
\end{proof}
\begin{lemma}\label{Lemma bound remainder}
Let $\lbrace e_j \rbrace_{j\geq 1} \subset E$ be any orthonormal basis of $H$. Then for any $\epsilon > 0$, there exists $k_{\epsilon}$ such that
\begin{align}
\sup_{x\in \Gamma_{\delta}}\big| \sum_{j\geq k_{\epsilon}} D^2 \pi_i(x)[B(x)e_{j},B(x)e_{j}] \big|\leq \epsilon .
\end{align}
\end{lemma}
\begin{proof}
Using the linearity of the derivative,
\begin{align}
D^2 \pi_i(x)[B(x)e_{j},B(x)e_{j}] = \sum_{k,l =1}^{\infty}\langle B(x)e_j , e_{\pi(x),k} \rangle \langle B(x)e_j , e_{\pi(x),l} \rangle  D^2 \pi_i(x)[ e_{\pi(x),k}, e_{\pi(x),l} ] 
\end{align}
It thus follows from Lemma \ref{Lemma bound decay basis vectors} that
\begin{align}
\big| D^2 \pi_i(x)[B(x)e_{j},B(x)e_{j}] \big| \leq \rm{Const} \sum_{k,l =1}^{\infty} \big| \langle B(x)e_j , e_{\pi(x),k} \rangle \langle B(x)e_j , e_{\pi(x),l} \rangle \big| \big( \lambda_k + \lambda_l \big)^{-1}.
\end{align}
The Lemma now follows from Assumption D. 

\end{proof}

The following Lemma is needed to prove that the stochastic integral in the isochronal SDE is well-defined.
\begin{lemma} \label{Lemma for the stochastic integral}
For any orthonormal basis $\lbrace e_j \rbrace_{j\geq 1} \subset E$ of $H$, and for any $\epsilon > 0$, there exists $k_{\epsilon}$ such that 
\begin{align}
\sup_{x\in B_{\delta}(\Gamma)} \sum_{k \geq k_{\epsilon}}\big| D\Xi^j(\pi(x),x)[ B(x)e_k ] \big|^2 \leq \epsilon
\end{align}
\end{lemma}
\begin{proof}
The expression in \eqref{eq: R definition} implies that $D\pi$ can be written as a linear combination of the elements $\lbrace D\Xi^j(\pi(x),x)[y]  \rbrace_{1\leq j \leq m-1} $. Furthermore the coefficients are uniformly bounded (as long as $\delta$ is sufficiently small). It thus suffices that we bound $\big| D\Xi^j(\pi(x),x)[y]  \big|$. To this end, we make the decomposition
\begin{align}
D\Xi^j(\alpha,x)[y] =& \big\langle  \psi^{*,j}_{ \alpha} , y \big\rangle + \tilde{\Xi}^j(\alpha,x)[y]  \text{ where }\\
\tilde{\Xi}^j(\alpha,x)[y] =& \int_0^{\infty}\big\langle  \psi^{*,j}_{\tilde{\phi}_s(\alpha)},  DN(\phi_s(x))\big[ D\phi_s(x)[y] \big] -DN\big( \gamma_{\tilde{\phi}_s(\alpha)}\big)\big[ D\phi_s(x)[y] \big]  \big\rangle ds .
\end{align}
It suffices that we prove that for any $\epsilon > 0$, we can find $k_{\epsilon} \in \mathbb{Z}^+$ such that
\begin{align}
\sup_{x\in B_{\delta}(\Gamma)} \sum_{k \geq k_{\epsilon}} \big| \tilde{\Xi}^j(\pi(x),x)[ B(x)e_k ] \big|^2 \leq \epsilon \\
\sup_{x\in B_{\delta}(\Gamma)} \sum_{k \geq k_{\epsilon}} \big\langle  \psi^{*,j}_{ \alpha} , B(x)e_k \big\rangle^2 \leq \epsilon
\end{align}
Let $\lbrace \alpha_{kl} \rbrace_{k,l \geq 1}$ be such that
\begin{align}
B(x)e_k = \sum_{l=1}^{\infty} \alpha_{kl}(x) e_{\pi(x),l}
\end{align}
Using the fact that $N$ is twice continuously differentiable, our Assumption \eqref{eq: bound psi inner product}, and the decay bound in \eqref{eq: decay linearization j again}, it must be that
\begin{align}
\big| \tilde{\Xi}^j(\pi(x),x)[ e_{\pi(x),l} ] \big| \leq \rm{Const} \times \int_0^\infty \exp\big( - \lambda_l t \big) dt = \rm{Const} \times \lambda_l^{-1}
\end{align}
and hence for some $k_{\epsilon} \in \mathbb{Z}^+$,
\begin{align}
\sup_{x\in B_{\delta}(\Gamma)} \sum_{k \geq k_{\epsilon}} \big| \tilde{\Xi}^j(\pi(x),x)[ B(x)e_k ] \big|^2 &\leq \rm{Const} \sum_{k \geq k_{\epsilon}}\bigg( \sum_{l=1}^\infty \alpha_{kl}(x) \lambda_l^{-1} \bigg)^2 \\
&= \rm{Const} \sum_{k \geq k_{\epsilon}}\bigg( \sum_{l=1}^\infty \big\langle B(x)e_k , e_{\pi(x),l} \big\rangle \lambda_l^{-1} \bigg)^2  \\
&= \rm{Const} \sum_{k \geq k_{\epsilon}}\sum_{l,j=1}^\infty \big\langle B(x)e_k , e_{\pi(x),l} \big\rangle \big\langle B(x)e_k , e_{\pi(x),j} \big\rangle  \lambda_l^{-1} \lambda_j^{-1} 
\end{align}
Thanks to Assumption E, the above can be made arbitrarily small by taking $k_{\epsilon}$ to be sufficiently large.
\end{proof}

\subsection{Ito Formula for the Isochronal Phase}

We let $\pi_t=\pi(X_t)$ be the isochronal phase of $X_t$  (as long as $X_t \in \Gamma_{\delta}$). In this subsection we obtain a specific SDE expression for $\pi_t$. Note that, unless additional assumptions are placed on $x\mapsto B(x)$, the noise in \eqref{eq:SDE}~may push the solution out of $B(\Gamma)$ in finite time.
We define the exit time from $\Gamma_\delta$ as 
\begin{equation}
\label{eq:tau_delta}
\tau \,\coloneqq\,\inf\left\{t>0\,:\,\norm*{X_t-\gamma_{\pi(X_t)}}_E=\delta\right\}. 
\end{equation}
(Furthermore we require $\delta$ to be sufficiently small that the regularity results in the first half of Section 3 all hold). This allows us to prove the following strong It{\^o}~formula for $\pi(X_t)$ that holds for $t<\tau$ (see \cite{antonopoulou2016motion}, \cite{IM16} and \cite{M22} for analogous proofs for the `variational phase'). 
\begin{thm}
\label{thm:Ito} 
For any orthonormal basis $\{e_k\}_{k\in\N}$ of $H$ contained in $E$, 
\begin{equation}
\label{eq:piIto}
\begin{aligned}
\pi(X_{t\wedge\tau})\,&=\, \pi(X_0) + \int_0^{t\wedge\tau} \tilde{N}\big( \pi(X_s) \big) \,ds + \int_0^{t\wedge\tau}\sigma^2 \sum_{k\in\N} D^2\pi(X_s)[B(X_s)e_k, B(X_s)e_k]\,ds \\
&\qquad + \sigma\int_0^{t\wedge\tau} D\pi(X_s)B(X_s)\,dW_s,   
\end{aligned} 
\end{equation}
and all of the above integrals are well-defined. (Note that $\tilde{N}: \mathcal{S} \to \mathcal{T}_{\mathcal{S}}$ is defined in \eqref{eq: tilde N definition}). 
\begin{proof}

First, we prove that the It{\^o}~integral 
\begin{equation}
\label{eq:piWint}
\int_0^{t\wedge\tau}D\pi(X_s)B(X_s)\,dW_s
\end{equation}
is well-defined.  Indeed, in more detail, if $\lbrace e_k \rbrace_{k\geq 1} \subset E$ is any orthonormal basis of $H$, and $\lbrace \beta^k_t \rbrace_{k\geq 1}$ are independent one-dimensional Brownian Motions, then the integral
\begin{equation}
\label{eq:piWint again}
\sum_{1\leq j \leq k} e_j  \int_0^{t\wedge\tau}D\pi(X_s)B(X_s)\,d\beta^j_s
\end{equation}
is well-defined for any $k\geq 1$. However thanks to Lemma \ref{Lemma for the stochastic integral}, \eqref{eq:piWint again} converges in mean square as $k\to \infty$, and the limit is by definition the integral in \eqref{eq:piWint}.

We now prove \eqref{eq:piIto}. 
Take $t>0$ and a partition of $[0,t]$ with boundary points $\{t_k\}_{k=1}^M$. 
Let $t^*\coloneqq t\wedge\tau$ and $t_i^*\coloneqq t_i\wedge\tau$ for $i\in\{1,\ldots,M\}$. 
By Theorem \ref{thm:C2}, $\pi(X_{t^*})$ may be written as 
\begin{equation}
\label{eq:PartitionTaylor} 
\pi(X_{t^*})\,=\,\pi(X_{t_0}) + \sum_{i=1}^M\Big( D\pi(X_{t^*_i})[X_{t^*_{i+1}}-X_{t^*_i}] + D^2\pi(w_i)[X_{t^*_{i+1}}-X_{t^*_i},X_{t^*_{i+1}}-X_{t^*_i}] \Big), 
\end{equation}
where $w_i = a_i X_{t^*_i} + (1-a_i)X_{t^*_{i+1}}$ for some $a_i\in[0,1]$. 
Since $(X_t)_{t\ge0}$ is a mild solution of \eqref{eq:SDE}, 
\[
\begin{aligned}
X_{t^*_{i+1}} - X_{t^*_i}\,&=\,\Lambda_{t^*_{i+1}-t^*_i}X_{t^*_i} - X_{t^*_i}+ \int_{t^*_i}^{t^*_{i+1}} \Lambda_{t^*_{i+1}-s}N(X_s)\,ds + 
 \sigma\int_{t^*_i}^{t^*_{i+1}} \Lambda_{t^*_{i+1}-s}B(X_s)\,dW_s \\ 
&\eqqcolon\, U^1_i + U^2_i + \sigma U_i^3. 
\end{aligned}
\]
Inserting this into \eqref{eq:PartitionTaylor}, we then have 
\begin{equation}
\label{eq:PartitionTaylorDecomp}
\begin{aligned}
\pi(X_{t^*})-\pi(X_0) \,&=\, \sum_{i=1}^M D\pi(X_{t^*_i})\left[ U^1_i \right] + \sum_{i=1}^M D\pi(X_{t^*_i})\left[ U^2_i \right] + \sigma\sum_{i=1}^M D\pi(X_{t^*_i})\left[ U_i^3\right] \\
&\quad + \sum_{i=1}^M D^2\pi(w_i)\left[U^1_i+U^2_i,U^1_i+U^2_i\right] +2\sigma\sum_{i=1}^M D^2\pi(w_i)\left[U^1_i+U^2_i,U_i^3\right] \\
&\qquad +\sigma^2\sum_{i=1}^M D^2\pi(w_i)\left[U^3_i,U^3_i\right] \\
&\eqqcolon I + II + III + IV + V + VI. 
\end{aligned}
\end{equation}
We will prove that, as the mesh size of the partition $\{t_i\}_{i=1}^M$ tends to zero, we have 
\begin{align}
(I+II)\,&\rightarrow\,\int_0^{t^*}\tilde{N}\big( \pi(X_s) \big) \,ds, \label{eq: to prove I II} \\
 III\,&\rightarrow\,\int_0^{t^*}D\pi(X_s)B(X_s)\,dW_s, \\
IV\,&\rightarrow\,0, \\ 
V\,&\rightarrow\,0, \\
VI\,&\rightarrow\,\int_0^{t^*}\sum_{k\in\N}D^2\pi(X_s)\left[B(X_s)e_k,B(X_s)e_k\right]\,ds. 
\end{align}
We first handle the nonzero limits. For a fixed partition $\{t_i\}_{i=1}^M$ with mesh size $h>0$, choose arbitrary $i\in\{1,\ldots,M\}$ and let $(\hat{X}_s)_{s\in[t_i,t_{i+1}]}$ be given by 
\[
\hat{X}_{t_i}\,=\,X_{t_i},\quad \hat{X}_s\,=\,\Lambda_{s-t_i}X_{t_i}+\int_{t_i}^{s}\Lambda_{t-r}N(\hat{X}_r)\,dr\quad\text{ for }\quad s\in[t_i,t_{i+1}]. 
\]
Using Corollary \ref{Cor isochronal phase invariant}, for $s\in [t_i,t_{i+1}]$,
\begin{align}
\pi\big( \hat{X}_s \big) &= \hat{\pi}(s) \text{ where } \\
\frac{d\hat{\pi}}{ds}(s) &= \tilde{N}\big( \hat{\pi}(s) \big) \text{ and } \hat{\pi}(t_i) = \pi\big( X_{t_i}\big).
\end{align}
Taking $h \to 0^+$, and employing our assumption that $\tilde{N}: \mathcal{S} \to \mathcal{T}_{\mathcal{S}}$ is continuous, we obtain \eqref{eq: to prove I II}.
For the limit of $III$, note that for small $h>0$ we have that, since $\norm{\Lambda_{t^*_{i+1}-s} - \mathcal{I} }= O(h)$,
\[
\int_{t^*_i}^{t^*_{i+1}}\Lambda_{t^*_{i+1}-s} B(X_s)\,dW_s\,=\, B(X_{t^*_i})[W_{t^*_{i+1}}-W_{t^*_i}] + o(h). 
\]
Hence, observing that $K=t/h$, we have 
\[
\begin{aligned}
III\,&=\, \sum_{i=1}^MD\pi(X_{t^*_i})\left[\int_{t^*_i}^{t^*_{i+1}} \Lambda_{t^*_{i+1}-s}B(X_s)\,dW_s\right]\\
&=\,\sum_{i=0}^K D\pi(X_{t^*_i})B(X_{t^*_i})[W_{t^*_{i+1}}-W_{t^*_i}] + o(h)\\
&\qquad\xrightarrow[h\rightarrow0]{}\,\int_0^{t^*}D\pi(X_s)B(X_s)\,dW_s. 
\end{aligned}
\]
To obtain $VI$, we formally decompose $(W_t)_{t\ge0}$ as 
\begin{equation}
\label{eq:WienerSum}
W_t=\sum_{k\in\N}e_k \beta^k_t 
\end{equation}
for some collection of independent identically distributed Brownian motions $\{ \beta^k_t \}_{k\in\N}$, and we recall that $\{e_k\}_{k\in\N}$ is the orthonormal basis of $H$ from Assumption {\bf D}.

We thus find that
\begin{align}
D^2\pi(w_i)\left[U^3_i,U^3_i\right] = D^2\pi(w_i)\left[\int_{t^*_i}^{t^*_{i+1}} \Lambda_{t^*_{i+1}-s}B(X_s)\sum_{k\in\N}e_k d\beta_s^k,\,\int_{t^*_i}^{t^*_{i+1}} \Lambda_{t^*_{k+1}-s}B(X_s)\sum_{\ell\in\N}e_\ell d\beta_s^\ell\right].
\end{align}
Now as $h\to 0$, it must be that for any $a \in \mathbb{Z}^+$,
\begin{align*}
\sum_{k,\ell \leq a} \sum_{i=1}^{M} D^2\pi(w_i)\left[\int_{t^*_i}^{t^*_{i+1}} \Lambda_{t^*_{i+1}-s}B(X_s)e_k d\beta_s^k,\,\int_{t^*_i}^{t^*_{i+1}} \Lambda_{t^*_{k+1}-s}B(X_s)\sum_{\ell\in\N}e_\ell d\beta_s^\ell\right] \\
\to \sum_{k\leq a} \int_0^{t\wedge \tau} D^2\pi(X_s)[ B(X_s)e_k, B(X_s)e_k ] ds.
\end{align*}
(One proves this by taking expectations of the square difference and proving that it goes to zero). Furthermore, thanks to Lemma \ref{Lemma bound remainder}, 
\begin{align}
\mathbb{E}\bigg[ \bigg\lbrace \sum_{k,\ell > a} \sum_{i=0}^{M-1} D^2\pi(w_i)\left[\int_{t^*_i}^{t^*_{i+1}} \Lambda_{t^*_{i+1}-s}B(X_s)e_k d\beta_s^k,\,\int_{t^*_i}^{t^*_{i+1}} \Lambda_{t^*_{k+1}-s}B(X_s)\sum_{\ell\in\N}e_\ell d\beta_s^\ell\right] \bigg\rbrace^2 \bigg] 
\end{align}
becomes arbitrarily small as $a \to \infty$. This implies VI.
Finally, one easily checks that $IV$ and $V$ tend to zero.

\end{proof}
\end{thm}

\section{Concentration \&~Metastability of the Isochronal Phase}
\label{sec:Ergodic}
In this section, we determine a measure $P_*$ that describes the average behaviour of $\pi_t\coloneqq\pi(X_t)$ on a particular large time scale in the small noise regime. 
The measure $P_*$ is explicitly determined (for small values of $\sigma$) in terms of the first and second derivatives of the isochron map and the diffusion coefficient $B$, as stated in Theorem \ref{thm:ApproximateErgodic}. 

We assume throughout this section $\Gamma$ consists of fixed points. That is, we assume that
\begin{equation}
\tilde{\phi}_t(\alpha) = \alpha \text{ for all } t\geq 0.
\end{equation}
We define the exit time from $\Gamma_\delta$ (the $\delta$ neighbourhood of $\Gamma$) to be 
\begin{equation}
\label{eq:tau_delta}
\tau_{\delta} \,\coloneqq\,\inf\left\{t>0\,:\,\norm*{X_t-\gamma_{\pi(X_t)}}_E=\delta\right\}, 
\end{equation}
and it is assumed that $\delta$ is sufficiently small that all of the regularity results of Section 3 hold. Recall that in \eqref{eq:Concentration}, we assumed that the following lemma holds. 

\begin{lemma}
\label{lemma:LDT2}
There exist constants $c>0$, $\sigma_*>0$, and $\delta_*>0$ such that if  $\sigma\in(0,\sigma_*)$ and $\delta\in(\sigma^2,\delta^*)$, then for $t>0$ it holds that 
\begin{equation}
\label{eq:LDT2}
\pr\left(t>\tau_\delta\right)\,\le\,t\exp\left(-c\sigma^{-2}\delta^2\right). 
\end{equation}
\end{lemma}
 Lemma \ref{lemma:LDT2} has been proved in \cite{M22} in the case that $\Gamma$ consists of stationary points, so long as the isochronal and variational phases are close to one another.  Such a closeness result is proven in Appendix \ref{sec:IsoVar}. 

In the proof of the next theorem, we demonstrate the existence of the measure $P_*$ on $\mathcal{S}$ alluded to above. The measure $P_*$ can be characterized as the unique invariant measure of a Markov process $\grave{\pi}_t$ on $\mathcal{S}$. To specify this process, for any $\alpha \in \mathcal{S}$, define $\mathcal{H}(\alpha) \in \mathbb{R}^{m-1 \times m-1}$ to be the square matrix, with components, for $1\leq j,k \leq m-1$,
\begin{align}
\mathcal{H}^{jk}(\alpha) &=\sum_{k\in \mathbb{N}}\inner*{ \mathcal{B}(\alpha)e_k}{ \mathfrak{v}_{\alpha,j}}  \inner*{\mathcal{B}(\alpha)e_k} {\mathfrak{v}_{\alpha,k}},
\end{align}
and $\lbrace e_k \rbrace_{k \geq 1} \subset E$ is any orthonormal basis of $H$, and in the above $\langle \cdot , \cdot \rangle$ is the Euclidean dot product over $\mathbb{R}^m$. It is immediate from the definition that $\mathcal{H}$ is symmetric and positive semi-definite, and we assume that there exists a non-zero lower bound for its eigenvalues, which holds uniformly for all $\alpha \in \mathcal{S}$. Write $G(\alpha) := \big(G_{jk}(\alpha)\big)_{1\leq j,k \leq m-1}$ to be the $\mathbb{R}^{(m-1) \times (m-1)}$ square matrix that is (i) co-diagonal with $H(\alpha)$, and (ii) with eigenvalues that are the positive square roots of the eigenvalues of $H(\alpha)$. Define $\grave{\pi}_t = (\grave{\pi}^i_t)_{1\leq i \leq m}$ to be the folllowing $\mathbb{R}^m$-valued stochastic process,
\begin{align}
d\grave{\pi}^i_t =& \mathcal{V}_i(\grave{\pi}_t)dt + 
 \sum_{1\leq j,k \leq m-1} \mathfrak{v}_{\grave{\pi}_t,j}^i G_{jk}(\grave{\pi}_t) dW^k_t \text{ where } \label{eq: hat pi stochastic process} \\
\mathcal{V}_i\big( \grave{\pi}_t \big) =& \frac{1}{2} \sum_{k=1}^{\infty}\sum_{j=1}^{m-1}\mathfrak{v}_{\grave{\pi}_t,j}^i D^2\pi_j(\gamma_{\grave{\pi}_t})[B(\gamma_{\grave{\pi}_t})e_k,B(\gamma_{\grave{\pi}_t})e_k] \label{eq: V vector field}
\end{align}
and $ \lbrace W^k_t \rbrace_{1\leq k \leq m-1}$ are independent $\mathbb{R}$-valued Brownian Motions, and $\grave{\pi}_0^i$ is any constant in $\mathcal{S}$. 
\begin{thm}
\label{thm:ApproximateErgodic}
With unit probability a modification of $\grave{\pi}_t$ stays on $\mathcal{S} \subset \mathbb{R}^m$ for all time. Furthermore, the process $\grave{\pi}_t$ possesses a unique invariant measure $P_* \in \mathcal{P}(\mathcal{S})$, such that the following holds: \\
Fix arbitrary $M>0$ and $\epsilon \ll 1$. There exist $\sigma_\epsilon,\,\delta_\epsilon,C_\epsilon>0$ (depending on both $\epsilon$ and $M$) such that if $\sigma \leq \sigma_{\epsilon}$, $\delta<\delta_\epsilon$, and $g\in C^2(\mathcal{S})$ has first and second derivatives bounded in Euclidean norm by $M$, then 
\begin{equation}
\label{eq:CorollaryBound}
\pr\left(\abs*{\frac{1}{t}\int_0^t g\big(\pi_s\big)\,ds - P_*(g)}<\epsilon,\,t<\tau_\delta\right)\,\ge\,1-t\exp\left(-c\sigma^{-2}\delta^2\right) - \exp\left(-C_\epsilon\sigma^2t\right),
\end{equation}
and we recall that $\pi_s = \pi(X_s)$.
\end{thm}
%

The following is then an immediate consequence. 

\begin{cor}
In the setting of Theorem \ref{thm:ApproximateErgodic}, if $(t_\sigma)_{\sigma>0}$ is any family of times such that 
\begin{equation}
\label{eq:t_sigma_1}
\sigma^2 t_{\sigma} \xrightarrow[\sigma\rightarrow0]{} \infty \quad\text{ and }\quad
t_\sigma \,\le\,\exp\left(c\sigma^{-2}\delta^2\right),  
\end{equation}
then 
\begin{equation}
\begin{aligned}
\pr\left(\abs*{\frac{1}{t_\sigma}\int_0^{t_\sigma} g\big(\pi_s\big)\,ds - P_*(g)}<\epsilon,\,t_\sigma<\tau_\delta\right)\,&\xrightarrow[\sigma\rightarrow0]{}\,1. 
\end{aligned}
\end{equation}
\end{cor}

Theorem \ref{thm:ApproximateErgodic}~indicates that, for each small $\sigma>0$, the measure $P_*$ is ``observable'' on the time scale indicated by \eqref{eq:t_sigma_1}, in the sense that on this time scale, a sufficient convergence to $P_*$ before escape from $\Gamma_\delta$ is observed with high probability. 
In Section \ref{sec:Spheres}, we demonstrate how one may use Theorem \ref{thm:ApproximateErgodic}~to compute the noise-induced drift of the isochronal phase when $\mathcal{S}=\mathbb{S}^1$. 


Since $\Gamma$ is assumed to consist of fixed points in this section, $\tilde{N} = 0$, and the It{\^o}~formula \eqref{eq:piIto} therefore assumes the form 
\begin{equation}
\label{eq:piItoRecentered}
d\pi_t\,=\,\frac{\sigma^2}{2}\sum_{k\in\N}D^2\pi(X_t)[B(X_t)e_k,B(X_t)e_k]\,dt + \sigma D\pi(X_t)B(X_t)\,dW_t. 
\end{equation}
Define  
\[
v_t\,\coloneqq\,X_t-\gamma_{\pi_t}, 
\]
so $\norm*{v_t}\lesssim\norm*{v_t}_E<\delta$ for $t<\tau$.
Thus, for $\delta>0$ sufficiently small, one expects that the following approximations are good for $t<\tau_\delta$: 
\begin{equation}
\label{eq:piApprox}
\pi(X_t)\,\simeq\,\pi(\gamma_{\pi(X_t)}),\quad D\pi(X_t)\,\simeq\,D\pi(\gamma_{\pi(X_t)}),\quad D^2\pi(X_t)\,\simeq\,D^2\pi(\gamma_{\pi(X_t)}). 
\end{equation}
We wish to explicitly identify a finite-dimensional diffusion on $\mathcal{S}$ whose probability law is very close to that of $\pi(X_t)$. Define
\begin{align}
\mathcal{B}(\alpha) &\in L(H,T_{\alpha}\mathcal{S}^*) \\
\mathcal{B}(\alpha)\,&\coloneqq\,D\pi(\gamma_\alpha)B(\gamma_\alpha).
\end{align}

We then consider the It{\^o}~diffusion on $\mathcal{S}$ governed by 
\begin{equation}
\label{eq:v}
d\hat{\pi}_t\,=\,\mathcal{V}(\hat{\pi}_t)\,dt + \mathcal{B}(\hat{\pi}_t)\,d\tilde{W}_t,\quad\hat{\pi}_0\,=\,\alpha, 
\end{equation}
where $\tilde{W}_t$ is a Wiener process that is equivalent in law to $W_t$ and independent of $W_t$. This stochastic process is identical (in probability distribution) to the process $\grave{\pi}_t$ in \eqref{eq: hat pi stochastic process}. The only difference is that \eqref{eq:v} is driven by the same infinite-dimensional Wiener Process as the original system, whereas $\grave{\pi}_t$ is driven by $(m-1)$ $\mathbb{R}$-valued Brownian Motions. When we want to emphasize that $\hat{\pi}$ has initial condition $\alpha\in\mathcal{S}$, we write $\hat{\pi_t}=\hat{\pi}_t(\alpha)$. 
\emph{A priori}, $\hat{\pi}_t$ is an $\R^{m}$-valued diffusion. 
However, one can verify that there is a version of this stochastic process such that for any $\alpha\in\mathcal{S}$ we have $\hat{\pi}_t(\alpha)\in\mathcal{S}$ for all $t\leq t_{\delta}$ almost surely. 
For further details see Remark V.1.1 of Ikeda \&~Watanabe \cite{IW14}. 

We also wish to consider a rescaled version of $\hat{\pi}$ defined on finite time intervals - doing this helps clarify the ergodic timescale for the induced phase dynamics on $\mathcal{S}$. 
Fix a time interval $\Delta=\Delta_{\sigma,\epsilon,\eta}>0$ (taken to be a precise value following the proof of Lemma \ref{lemma:H}, below), and let $t_a\coloneqq a\Delta$ for $a\in\N$ (note that the sequence $\{t_a\}_{a\in\N}$ used here is distinct from that in Section \ref{sec:Regularity}). 
Then, define $\tilde{\pi}^a_t$ for $t\in[t_a,t_{a+1}]$ as the solution to 
\begin{equation}
\label{eq:vv}
d\tilde{\pi}^a_t\,=\,\sigma^2\mathcal{V}(\tilde{\pi}^a_t)\,dt + \sigma\mathcal{B}(\tilde{\pi}^a_t)\,dW_t, \quad\tilde{\pi}_{t_a}^a\,=\,\pi_{t_a}, 
\end{equation}
where $W_t$ in \eqref{eq:vv}~is pathwise identical to the process in \eqref{eq:SDE}. 
Note that $\tilde{\pi}^a_t$ is identical in law to $\hat{\pi}_{\sigma^2t}(\pi_{t_a})$ for each $t\in[t_a,t_{a+1}]$, due to the fact that the two Markov processes have the same infinitesimal generator. 

The remainder of this section is as follows. 
In Lemma \ref{lemma:H}, we show that $P_*$ exists as the unique invariant measure of \eqref{eq:v}~to which the distribution of $\hat{\pi}_t$ converges uniformly. 
In Lemma \ref{lemma:LDT1}, we then show that $\tilde{\pi}_t^a$ closely approximates $\pi_t$ on each $[t_a,t_{a+1}]$. 
Using this, we demonstrate that \eqref{eq:CorollaryBound}~holds, completing the proof of Theorem \ref{thm:ApproximateErgodic}.

\subsection{Proof of Theorem \ref{thm:ApproximateErgodic}}
\label{sec:Proof} 
In this section, we make extensive use of the regularity of the isochron map provided for in Section \ref{sec:Regularity}. 
To begin, we note that the existence of a unique, uniformly convergent invariant measure of \eqref{eq:v}~is a well-established result under certain assumptions, which we state in the following lemma. 

\begin{lemma}
\label{lemma:H}
The system \eqref{eq:v}~possesses a unique invariant measure $P_*$ on $\mathcal{S}$ to which it converges at an exponential rate in total variation norm. 
Moreover, for arbitrary $\epsilon,\eta>0$, there exists a constant $\tilde{\Delta}=\tilde{\Delta}_{\epsilon,\eta}>0$ such that if $t\ge \tilde{\Delta}$ and $g\in C^2(\mathcal{S})$ then 
\begin{equation}
\label{eq:Lem1}
\sup_{\alpha\in \mathcal{S}}\pr\left(\abs*{\frac{1}{t}\int_0^t g(\hat{\pi}_s(\alpha))\,ds - P_*(g)}\ge\epsilon/2\right)\,\le\,\eta. 
\end{equation}
\begin{proof}
It is well-known that \eqref{eq:v}~possesses a unique invariant measure such that \eqref{eq:Lem1}~holds. This follows from the fact that $\mathcal{S}$ is compact and the diffusion coefficient is uniformly bounded from below, which implies that $\grave{\pi}_t$ must satisfy a Doeblin Condition. See for instance \cite[Section 1.10.1]{BGL14}. 
\end{proof}
\end{lemma}

We henceforth fix $\Delta=\Delta_{\sigma,\epsilon,\eta}\coloneqq\sigma^{-2}\tilde{\Delta}_{\epsilon,\eta}$, and let $t_a=a\Delta$ for $a\in\N$. 
Below, we show that $P_*$ in Lemma \ref{lemma:H}~satisfies the properties described in Theorem \ref{thm:ApproximateErgodic}. 
To this end, the following lemma is needed, which places a bound on the difference between $\pi_t$ and $\tilde{\pi}^a_t$ on each $[t_a,t_{a+1}]$. 
From these bounds on time intervals of length $\Delta$, we will obtain the long time bounds in 
Theorem \ref{thm:ApproximateErgodic}. 

\begin{lemma}
\label{lemma:LDT1}
For any $\epsilon,\eta>0$, there exist $\delta_0,\sigma_0>0$ such that for $\sigma<\sigma_0$ and $\delta<\delta_0$, 
\begin{equation}
\label{eq:LDT1}
\pr\left(\sup_{t\in[t_a,t_{a+1}]}\norm*{\pi_t-\tilde{\pi}^a_t}\ge\epsilon,\,\sup_{t\in[t_a,t_{a+1}]}\norm*{v_t}\le\delta\right)\,<\,\eta. 
\end{equation}
\begin{proof}
Using the Lipschitz continuity of the map $\phi$, and the Lipschitz continuity of $\mathcal{V}$, we observe that so long as $v_t<\delta$ we have  
\begin{equation}
\begin{aligned}
\norm*{\pi_t-\tilde{\pi}_t^a}\,&\le\,\frac{\sigma^2}{2}\norm*{\int_{t_a}^t\sum_{k\in\N}D^2\pi(X_s)[B(X_s)e_k,B(X_s)e_k] - \mathcal{V}(\tilde{\pi}_s^a)\,ds} \\
&\qquad\quad+\sigma\norm*{\int_{t_a}^tD\pi(X_s)B(X_s) - \mathcal{B}(\tilde{\pi}_s^a)\,dW_s}\\
&\le\, \frac{\sigma^2}{2}\norm*{\int_{t_a}^t\sum_{k\in\N}D^2\pi(X_s)[B(X_s)e_k,B(X_s)e_k] - \sum_{k\in\N}D^2\pi(\gamma_{\pi_s})[B(\gamma_{\pi_s})e_k,B(\gamma_{\pi_s})e_k]\,ds} \\
&\qquad\quad + \frac{\sigma^2}{2}\norm*{\int_{t_a}^t\mathcal{V}(\pi_s) - \mathcal{V}(\tilde{\pi}_s^a)\,ds} \\
&\qquad\qquad\quad + \sigma\norm*{\int_{t_a}^t D\pi(X_s)B(X_s)-D\pi(\gamma_{\tilde{\pi}_s^a})B(\gamma_{\tilde{\pi}_s^a})\,dW_s} \\
&\le\,\frac{\sigma^2}{2}C_\pi\norm*{\int_{t_a}^t X_s-\gamma_{\pi_s}\,ds} + \frac{\sigma^2}{2}C_\mathcal{V}\norm*{\int_{t_a}^t\pi_s-\tilde{\pi}_s^a\,ds} \\
&\qquad\quad+ \sigma\norm*{\int_{t_a}^t D\pi(X_s)B(X_s)-D\pi(\gamma_{\tilde{\pi}_s^a})B(\gamma_{\tilde{\pi}_s^a})\,dW_s}\\
&\le\, \frac{\sigma^2}{2}C_\pi\Delta\delta +\sigma\norm*{\int_{t_a}^t D\pi(X_s)B(X_s)-D\pi(\gamma_{\tilde{\pi}_s^a})B(\gamma_{\tilde{\pi}_s^a})\,dW_s} + \frac{\sigma^2}{2}C_\mathcal{V}\int_{t_a}^t\norm*{\pi_s-\tilde{\pi}_s^a}\,ds \\
&\eqqcolon\, K_0 + \sigma\norm*{\mathfrak{B}_t} + K_1\int_{t_a}^t\norm{\pi_s-\tilde{\pi}_s^a}\,ds. 
\end{aligned}
\end{equation}
Then, by Gr{\"o}nwall's inequality, we have 
\begin{equation}
\begin{aligned}
\norm*{\pi_t-\tilde{\pi}_t^a}\,&\le\,\sigma\norm*{\mathfrak{B}_t}+K_0 + \int_{t_a}^t(\sigma\norm*{\mathfrak{B}_s}+K_0)K_1e^{K_1(t-s)}\,ds \\
&=\,\sigma\norm*{\mathfrak{B}_t} + K_0 + K_0K_1\int_{t_a}^te^{K_1(t-s)}\,ds + \sigma K_1\int_{t_a}^t\norm*{\mathfrak{B}_s}e^{K_1(t-s)}\,ds \\
&\le\, \sigma\norm*{\mathfrak{B}_t} + K_0 + K_0(e^{K_1(t-t_a)}-1) +\sigma K_1\int_{t_a}^t\norm*{\mathfrak{B}_s}e^{K_1(t-s)}\,ds \\
&\le\,\sigma\norm*{\mathfrak{B}_t} +  K_0e^{K_1\Delta} +\sigma K_1\int_{t_a}^t\norm*{\mathfrak{B}_s}e^{K_1\Delta}\,ds. 
\end{aligned}
\end{equation}
It therefore follows that, so long as $\norm*{v_t}\le\delta$, 
\[
\begin{aligned}
\sup_{t\in[t_a,t_{a+1}]}\norm*{\pi_t-\tilde{\pi}_t^a}\,&\le\,\sigma
\sup_{t\in[t_a,t_{a+1}]}\norm*{\mathfrak{B}_t} + K_0e^{K_1\Delta} + \sigma K_1\Delta e^{K_1\Delta}
\sup_{t\in[t_a,t_{a+1}]}\norm*{\mathfrak{B}_t}\\
&\eqqcolon\, K_0e^{K_1\Delta} + K_2
\sup_{t\in[t_a,t_{a+1}]}\norm*{\mathfrak{B}_t}. 
\end{aligned}
\]
Then, by a union of events bound we find that 
\begin{equation}
\label{eq:UnionOfEvents} 
\begin{aligned}
\pr\left(\sup_{t\in[t_a,t_{a+1}]}\norm*{\pi_t-\tilde{\pi}_t^a}\ge\epsilon,\,\,\,\norm*{v_t}\le\delta\right)\,&\le\,
\pr\left(K_0e^{K_1\Delta}\ge\frac{\epsilon}{2}\right) + \pr\left(\sup_{t\in[t_a,t_{a+1}]}\norm*{\mathfrak{B}_t} \ge\frac{\epsilon}{2K_2}\right). 
\end{aligned}
\end{equation}

For any $\epsilon,\eta,\sigma>0$, we can take $\delta_0>0$ small enough such that for $\delta<\delta_0$ we have 
\begin{equation}
\label{eq:ExplicitBounds}
K_0e^{K_1\Delta}\,=\,
\delta \frac{C_\pi}{2}\sigma^2\Delta e^{{\frac{C_{\mathcal{V}}}{2}\sigma^2\Delta}}\,<\, \frac{\epsilon}{2}.
\end{equation}
Indeed, recalling that $\Delta=O(\sigma^{-2})$, we may choose $\delta_0$ independently of $\sigma$. 
This implies that the first term in the right hand side of \eqref{eq:UnionOfEvents}~is zero. 
To handle the second term, fix $\delta<\delta_0$. 
Then, we apply Markov's inequality for third moments, and the Burkholder-Davis-Gundy inequality to find that 
\[
\begin{aligned}
\pr\left(\sup_{t\in[t_a,t_{a+1}]}\norm*{\mathfrak{B}_t} \ge\frac{\epsilon}{2K_2}\right)\,&\le\,
\frac{8K_2^3}{\epsilon^3}\E{\left(\sup_{t\in[t_a,t_{a+1}]}\norm*{\mathfrak{B_t}}\right)^3}\\
&=\,\frac{8K_2^3}{\epsilon^3}\E{\sup_{t\in[t_a,t_{a+1}]}\norm*{\int_{t_a}^{t} D\pi(X_s)B(X_s) - D\pi(\gamma_{\pi_s})B(\gamma_{\pi_s})\,dW_s }^3}\\
&\le\, \frac{8K_2^3}{\epsilon^3}\E{\sum_{k\in\N}\sup_{t\in[t_a,t_{a+1}]}\norm*{\int_{t_a}^{t}\big[D\pi(X_s)B(X_s) - D\pi(\gamma_{\pi_s})B(\gamma_{\pi_s})\big]e_k\,d\beta_\mathbb{S}^k}^3}\\
&\qquad\qquad + R_t \\ 
&\le\, \frac{8K_2^3}{\epsilon^3}C_3\E{\sum_{k\in\N}\int_{t_a}^{t_{a+1}}\norm*{\big[D\pi(X_s)B(X_s) - D\pi(\gamma_{\pi_s})B(\gamma_{\pi_s})\big]e_k}^{3/2}\,ds}\\
&\le\, \frac{16K_2^3}{\epsilon^3}C_3\Delta \sup_{x\in\Gamma_\delta}\sum_{k\in\N}\norm*{\big[D\pi(x)B(x)e_k}^{3/2}\\
&\le\, \big(16C_3M_\pi\big)K_2^3\Delta\epsilon^{-3}, 
\end{aligned}
\]
where $C_3$ is the constant appearing in the Burkholder-Davis-Gundy inequality. 
The remainder term $R_t$ in the third line is the off diagonal terms appearing in the expression for the cube of a sum, and can be identified with zero. 
Since $K_2=O(\sigma)$ and $\Delta=O(\sigma^{-2})$, we see that for arbitrary $\epsilon,\eta>0$, we can obtain 
\[
\pr\left(\sup_{t\in[t_a,t_{a+1}]}\norm*{\mathfrak{B}_t} \ge\frac{\epsilon}{2K_2}\right)\,\le\,\frac{\eta}{2} 
\]
by taking $\sigma$ less than or equal to some $\sigma_0$ proportional to $\epsilon^{3}\eta$. 
Note that $\sigma_0$ depends on $\delta_0$ via $M_\pi$. 
This completes the proof. 
\end{proof}
\end{lemma}

In the next lemma we `combine' the previous two lemmas, showing that the distribution of $\pi_t$ becomes arbitrarily close to that of $P_*$ over sufficiently long time intervals. 

\begin{lemma}\label{Lemma tilde T epsilon eta definition}
For any $\epsilon,\eta > 0$ and $a\in\N$, for $\delta<\delta_0$ and sufficiently small $\sigma>0$ we have 
\begin{equation}
\label{eq:DeltaDef}
\begin{aligned}
\pr\left(\abs*{\Delta^{-1}\int_{t_a}^{t_{a+1}}g(\pi_s)\,ds - P_*(g)}\ge\epsilon,\,\tau_{\delta}\ge t_{a+1}\right) \leq \eta.  
\end{aligned}
\end{equation}
\end{lemma}
\begin{proof}
Using a union of events bound, 
\[
\begin{aligned}
&\pr\left(\abs*{\Delta^{-1}\int_{t_a}^{t_{a+1}}g(\pi_s)\,ds - P_*(g)}\ge\epsilon ,\,\tau_{\delta}\ge t_{a+1}\right)\\
&\le\,\pr\left(\abs*{\Delta^{-1}\int_{t_a}^{t_{a+1}} g(\tilde{\pi}_s)\,ds - \,P_*(g)}\,\ge\,\epsilon\right) + \pr\left(\Delta^{-1}\int_{t_a}^{t_{a+1}}\abs*{g(\pi_s)-g(\tilde{\pi}_s)}\,ds\ge\epsilon,\,\tau_{\delta}\ge t + \Delta\right).
\end{aligned}
\] 
Since $\hat{\pi}_{\sigma^2(t-t_a)}$ with initial condition $\pi_{t_a}$ has an identical probability law to $\tilde{\pi}_{t-t_a}^a$, for any $\eta>0$ 
\[
\begin{aligned}
\pr\left(\abs*{\frac{1}{\Delta }\int_{t_a}^{t_{a+1}} g(\tilde{\pi}_s)\,ds - P_*(g)}\ge\epsilon/2\right)\,&=\,\pr\left(\abs*{\frac{1}{\Delta \sigma^{-2}}\int_0^{\Delta \sigma^{-2}} g(\hat{\pi}_s)\,ds - P_*(g)}\ge\epsilon/2\right)\\
&\le\,\eta, 
\end{aligned}
\]
so long as $\sigma^{-2}\Delta_{\sigma,\epsilon,\eta}\geq \tilde{\Delta}_{\epsilon/2,\eta}$, by \eqref{eq:Lem1}, which can be achieved by taking $\sigma$ is sufficiently small. 
Now, since the Lipschitz constant of $g$ is bounded by $m$, Lemma \ref{lemma:LDT1}~implies that there exists $\sigma_0>0$ such that if $\sigma<\sigma_0$
then 
\[
\begin{aligned}
\pr\left(\frac{1}{\Delta}\int_{t_a}^{t_{a+1}} \abs*{g(\pi_s)-g(\tilde{\pi}_s)}\,ds \ge \epsilon/2,\,\tau_\delta>t_{a+1}\right)\,&\le\,\pr\left(\sup_{t\in[t_a,t_{a+1}]}\norm*{\pi_t-\tilde{\pi}_t}\ge\epsilon/(2m),\,\tau_\delta\ge t_{a+1}\right)\\
&\le\,\eta. 
\end{aligned}
\]
\end{proof}

We now combine the above events to prove a bound on the probability of the event ``either the average of $g(\pi_s)$ is close to $P_*$, or $t>\tau_\delta$''. 
Hence, the proof of Theorem \ref{thm:ApproximateErgodic}~follows from a simple property of probability measures. 

\begin{lemma}
\label{lemma:ApproximateErgodic}
Fix arbitrary $M>0$ and small $\epsilon>0$. 
There exist $\sigma_\epsilon,\,\delta_\epsilon,C_\epsilon>0$ (depending on both $\epsilon$ and $M$) such that if $\sigma \leq \sigma_{\epsilon}$, $\delta<\delta_\epsilon$, and $g\in C^2(\mathcal{S})$ has first and second derivatives bounded in Euclidean norm by $M$, then 
\begin{equation}
\pr\left(\left|\frac{1}{t}\int_0^t g\big(\pi_s\big)\,ds - P_*(g)\right|>\epsilon,\,\,t<\tau_\delta \right)\,\le\,\exp\left(-C\sigma^2t\right). 
\label{eq:FirstOrderErgodic1}
\end{equation}
\begin{proof}
Without loss of generality, take $g$ to be non-negative and uniformly bounded above by $1$.
Following Lemma \ref{lemma:H}, the convergence to ergodic averages of \eqref{eq:v}~is uniform. For $\sigma,\epsilon,\eta > 0$, let $\Delta=\sigma^{-2}\tilde{\Delta}_{\epsilon,\eta}$ be as in Lemma \ref{Lemma tilde T epsilon eta definition}, and for $a\in\N$ set $t_a\coloneqq a\Delta$. By Chernoff's bound, for arbitrary $\kappa>0$ and $t>\Delta$, the left hand side of \eqref{eq:FirstOrderErgodic1}~may be bounded as 
\begin{equation}
\label{eq: preliminary last bound}
\begin{aligned}
&\pr\left(\left|\frac{1}{t}\int_0^{t}g\big(\pi_s\big)\,ds - P_*(g)\right|>\epsilon,\,\,t<\tau_\delta\right)\,\le\,\\
&\qquad\qquad\qquad\ExpOp\bigg[\chi(\tau_\delta\ge t)\bigg(\exp\left(\kappa\int_0^t g(\pi_s)\,ds - \kappa t P_*(g) - \kappa t\epsilon\right) \\
&\qquad\qquad\qquad\qquad\qquad\qquad+ \exp\left(-\kappa\int_0^{t}g(\pi_s)\,ds + \kappa t P_*(g) - \kappa \epsilon t\right)\bigg)\bigg]. 
\end{aligned}
\end{equation}
We only bound the first term on the right hand side (the bound of the other term is very similar). We are going to show that there exists a constant $\tilde{\kappa}_{\epsilon}$ (depending on $\epsilon$ but not $\sigma$) such that
\begin{equation}
\label{eq:Eexp_bound}
\E{\chi(\tau_\delta\ge t)\exp\left(\tilde{\kappa}_{\epsilon}\sigma^{2}\int_0^{t}g(\pi_s)\,ds - \tilde{\kappa}_\epsilon\sigma^{2}t P_*(g)\right)}\,\le\,\exp\big(\epsilon \tilde{\kappa}_{\epsilon} \sigma^2 t/2\big). 
\end{equation}
Once we have established \eqref{eq:Eexp_bound}, the result \eqref{eq:FirstOrderErgodic1} follows from substituting $\kappa = \tilde{\kappa}_\epsilon\sigma^{2}$ into \eqref{eq: preliminary last bound}. 
Towards our goal of proving \eqref{eq:Eexp_bound}, define the random variable 
\[
H_a\,\coloneqq\,\chi(\tau_\delta\ge t_{a+1})\exp\left(\kappa\left(\int_{t_a}^{t_{a+1}} g(\pi_s)\,ds - \kappa (t_{a+1} - t_{a}) P_*(g)\right)\right), 
\]
and define $\hat{\mathfrak{a}}$ as the random index such that $\tau_\delta\in[t_{\hat{\mathfrak{a}}},t_{\hat{\mathfrak{a}}+1}]$. 
Then, 
\begin{equation} \label{eq: iterate H a}
\begin{aligned}
&\E{\chi(\tau_\delta\ge t)\exp\left(\kappa\int_0^{t} g(\pi_s)\,ds - \kappa\tau_\delta P_*(g)\right)}\\
&\qquad=\,\E{\chi(\tau_\delta\ge t)\prod_{a=0}^{\hat{\mathfrak{a}}} H_a\times\exp\left(\kappa\int_{t_{\hat{\mathfrak{a}}}}^{\tau_{\delta}}g(\pi_s)\,ds - \kappa(\tau_\delta-t_a) P_*(g)\right)}\\
&\qquad\le\, \E{\chi(\tau_\delta\ge t)\prod_{a=0}^\mathfrak{a} H_a}\exp(2\kappa \Delta), 
\end{aligned}
\end{equation}
and $\mathfrak{a}\coloneqq\floor{t/\Delta}-1$. 
Taking the conditional expectation on events up to the stopping time, we have 
\begin{equation}
\mathbb{E}\bigg[\chi(\tau_\delta\ge t)\exp(-\kappa\tau_\delta\epsilon)\prod_{a=0}^{\mathfrak{a}} H_a\bigg] = \mathbb{E}\bigg[ \prod_{a=1}^{\mathfrak{a}-1} H_a \mathbb{E}\big[ \chi\big\lbrace \tau_{\delta} \geq t \big\rbrace H_{\mathfrak{a}} | \mathcal{F}_{t_{\mathfrak{a}-1}} \big] \bigg].
\end{equation}
We first bound the conditional expectation. Write
\[
q = \int_{t_{\mathfrak{a}}}^{t_{\mathfrak{a}+1}} g(\pi_s)ds - \big\lbrace t_{\mathfrak{a}+1} - t_{\mathfrak{a}} \big\rbrace P_*(g). 
\]
Then, assuming that $\tau_{\delta} \geq t_{\mathfrak{a}-1}$, for any $\alpha > 0$,
\begin{align}\label{eq: expectation decomposition}
\mathbb{E}\big[ \chi\big\lbrace \tau_{\delta} \geq t \big\rbrace H_{\mathfrak{a}} | \mathcal{F}_{t_{\mathfrak{a}-1}} &\big] = \mathbb{E}\big[ \chi\big\lbrace \tau_{\delta} \geq t \text{ and }q \geq \alpha \big\rbrace H_{\mathfrak{a}} | \mathcal{F}_{t_{\mathfrak{a}-1}} \big] + \mathbb{E}\big[ \chi\big\lbrace \tau_{\delta} \geq t \text{ and }q < \alpha \big\rbrace H_{\mathfrak{a}} | \mathcal{F}_{t_{\mathfrak{a}-1}} \big]\nonumber \\
&\leq \mathbb{P}\big( q \geq \alpha, \tau_{\delta} \geq t_{\mathfrak{a}+1}\big) \exp(\kappa \Delta) + \mathbb{P}(q < \alpha,\tau_{\delta} \geq t_{\mathfrak{a}+1})\exp(\alpha \kappa \Delta ) \nonumber\\
&\leq \mathbb{P}\big( q \geq \alpha, \tau_{\delta} \geq t_{\mathfrak{a}+1}\big) \big\lbrace \exp( \kappa \Delta) - \exp(\alpha \kappa \Delta ) \big\rbrace + \exp(\alpha \kappa \Delta ),
\end{align}
since evidently 
\[
\mathbb{P}\big( q \geq \alpha, \tau_{\delta} \geq t_{\mathfrak{a}+1}\big) + \mathbb{P}\big( q < \alpha, \tau_{\delta} \geq t_{\mathfrak{a}+1}\big) \leq 1.
\]
Now, thanks to Lemma \ref{Lemma tilde T epsilon eta definition}, we can take $\eta$ and $\alpha=\epsilon$ to be such that
\[
\mathbb{P}\big( q\geq \epsilon , \tau_{\delta} \geq t_{\mathfrak{a}} \big) \leq \eta,
\]
and we thus obtain from \eqref{eq: expectation decomposition} that
\begin{align}
\E{\chi\big\lbrace \tau_{\delta} \geq t \big\rbrace H_{\mathfrak{a}} | \mathcal{F}_{t_{\mathfrak{a}-1}}} &\leq \exp\big\lbrace \epsilon \Delta + \eta \exp(\kappa \Delta ) - \eta \exp(\epsilon \Delta ) \big\rbrace \nonumber  \\
&\leq \exp\big\lbrace \epsilon\kappa \Delta + 2\eta \kappa \Delta  \big\rbrace ,
\end{align}
as long as $\kappa \Delta$ is sufficiently small (by taking a first order Taylor expansion). 
Since $\Delta \simeq O(\sigma^{-2})$, it suffices to take $\kappa = \sigma^2 \tilde{\kappa}_{\epsilon}$, where $\tilde{\kappa}_{\epsilon} > 0$ is a constant independent of $\sigma$ such that $\tilde{\kappa}_{\epsilon} \to 0$ as $\epsilon \to 0$. 
We thereby obtain 
\begin{equation}
\mathbb{E}\big[ \chi\big\lbrace \tau_{\delta} \geq t \big\rbrace H_{\mathfrak{a}} | \mathcal{F}_{t_{\mathfrak{a}-1}} \big] \leq \exp\big( \epsilon \kappa \Delta +  2\eta\kappa \Delta \big). 
\end{equation}
It thus follows from \eqref{eq: iterate H a} that
\begin{align}
&\E{\chi(\tau_\delta\ge t)\exp\left(\kappa\int_0^{t} g(\pi_s)\,ds - \kappa\tau_\delta P_*(g)\right)} \leq 
 \mathbb{E}\bigg[ \prod_{a=1}^{\mathfrak{a}-1} H_a  \bigg] \exp\big(\epsilon \kappa \Delta + 2\eta \kappa \Delta \big).
\end{align}
We iterate this inequality, from $a= \mathfrak{a}$ down to $a=1$, and obtain that
\begin{align}
\E{\chi(\tau_\delta\ge t)\exp\left(\kappa\int_0^{t} g(\pi_s)\,ds - \kappa\tau_\delta P_*(g)\right)}
\leq & \exp\big( (\mathfrak{a}+1)(2 \eta + \epsilon ) \kappa \Delta \big) \nonumber \\
\leq &  \exp\big( \kappa (2\eta  + \epsilon) \kappa t / \Delta  \big),
\label{eq:Problematic}
\end{align}
where we have substituted $\mathfrak{a}\coloneqq\floor{t/\Delta}-1$ in the penultimate line. 

Substituting \eqref{eq:Problematic}~into \eqref{eq: preliminary last bound}~and using $\kappa/\Delta=\tilde{\kappa}_\epsilon/\tilde{\Delta}<1$ for sufficiently small $\epsilon>0$ yields 
\[
\pr\left(\abs*{\frac{1}{t}\int_0^{t}g(\pi_s)\,ds-P_*(g)}>\epsilon,\,t<\tau_\delta\right)\,\le\,2\exp\left(\sigma^2\tilde{\kappa}\left(\frac{2\tilde{\kappa}}{\tilde{\Delta}}\eta + \left(\frac{\tilde{\kappa}}{\tilde{\Delta}} - 1\right)\epsilon\right)t\right).
\]
Since $\eta$ is a free parameter used in this proof only to determine the length of the time intervals $\Delta$, we may take $\eta$ to be small enough for the constant within the second exponential above to be negative. 
We have thus proven that, for arbitrary $t>0$ and any $\epsilon>0$, there exists a constant $C=C_{\epsilon,\eta}$, which can be taken to be positive such that 
\begin{equation}
\label{eq:Result}
\pr\left(\abs*{\frac{1}{t}\int_0^t g(\pi_s)\,ds-P_*(g)}>\epsilon,\,t<\tau_\delta\right)\,\le\, 2\exp\left(-C\sigma^2t\right). 
\end{equation}
\end{proof}
\end{lemma}

Lemma \ref{lemma:ApproximateErgodic}~then allows us to prove Theorem \ref{thm:ApproximateErgodic}~using basic facts from probability theory. 

\begin{proof}[Proof of Theorem \ref{thm:ApproximateErgodic}]
Define the events 
\[
\begin{aligned}
F_1\,&\coloneqq\,\left\{\abs*{\frac{1}{t}\int_0^t g\big(\pi_s\big)\,ds - P_*(g)}<\epsilon\right\}, \\
F_2\,&\coloneqq\,\left\{t<\tau_\delta\right\}. 
\end{aligned}
\]
Then, using Theorem \ref{lemma:LDT2}~and Lemma \ref{lemma:ApproximateErgodic}, observing that 
\[
\pr(F_1\cap F_2)=\pr(F_2)-\pr(F_1^c\cap F_2)  
\]
completes the proof. 
\end{proof}

\subsection{Noise-Induced Drift on a Circle}
\label{sec:Spheres}
In this section, we consider the example where $\mathcal{S}$ is the circle $\mathbb{S}^1$, 
such as when $\Gamma$ is a limit-cycle.  Defining the following vector field $\mathcal{Z}: E \to \mathbb{R}^m$,
\begin{align}
\mathcal{Z}(X) =  \sum_{k\in\N} D^2\pi(X)[B(X)e_k, B(X)e_k]
\end{align}
we assume throughout this section that for small enough $\delta$, and all $X,Y \in \Gamma_{\delta}$, there is a constant such that
\begin{align}\label{eq: assumption continuity vector field}
\norm{\mathcal{Z}(X) - \mathcal{Z}(Y)}_E  \leq \rm{Const} \norm{X-Y}_E.
\end{align}

Giacomin, Poquet and Shapira \cite{GPS18} have obtained an analogous result to this for finite-dimensional limit cycles. Infinite dimensional examples that this result applies to include the noise-induced wandering of neural bumps in the visual cortex \cite{kilpatrick2013wandering,maclaurin2020wandering}, and noise-induced deviation in the speed of travelling waves in the Fitzhugh-Nagumo system on a periodic spatial domain \cite{EGK20}.

Rather than viewing $\mathbb{S}^1$ as a subset of $\R^2$, in this section we identify $\mathbb{S}^1$ with the half open interval $[0,2\pi)$. 
For any continuous function $f:\R\rightarrow [0,2\pi)$ (this interval being endowed with the natural topology on $\mathbb{S}^1$ that identifies the endpoints), we let $\overline{f}:\R\rightarrow\R$ be the unique continuous function such that 
\[
f(t)\,=\,\overline{f}(t)\,\,\text{mod}\,2\pi. 
\]
In more concrete terms, $\overline{\pi}_t$ is the arc-length covered by the isochronal phase at time $t>0$. (If one wants more details, refer to Propositions 1.33 \&~1.34 of Hatcher \cite{Ha02}.)

In this section, we emphasize the dependence of the isochronal phase on $\sigma\ge0$ by writing $\pi_t=\pi_t^\sigma$
Note that over any closed subinterval $B_{\zeta} = [\zeta, 2\pi - \zeta ] \subset [0,2\pi)$ for $\zeta > 0$, we have 
\begin{equation}
\label{eq:overlinepi}
d\overline{\pi}_t^\sigma\,=\,d\pi_t^\sigma, 
\end{equation}
where $d\pi_t^\sigma$ is as in \eqref{eq:piIto}. Now write, for $t>0$, 
\[
\nu_t(B_{\zeta}) = \rm{Lebesgue}\big( \lbrace s\le t: \pi_s \notin B_{\zeta} \rbrace \big).
\]
As $\zeta \to 0^+$, it must be that $\nu_t(B^c_{\zeta}) \to 0$ almost surely. Furthermore the measure $P_*$ from Theorem \ref{thm:ApproximateErgodic}~can be interpreted as a measure on $[0,2\pi)$. 
This motivates the following corollary, which implies Corollary D of Section \ref{sec:Results}~just as in the proof of Theorem \ref{thm:ApproximateErgodic}. 

\begin{cor}
\label{cor:Spheres}
Suppose that $\mathcal{S} = \mathbb{S}^1$ and that the vector field $\mathcal{V}$ defined in \eqref{eq: V vector field}~is $C^2$. 
Then 
\begin{equation}
\pr\left(\norm*{\frac{1}{\sigma^2t}\left(\overline{\pi}_{t}^\sigma-\overline{\pi}_t^0\right) - P_*\left(\mathcal{V}\right) }<\epsilon,T_{\sigma}<\tau_\delta\right)\,\ge\,1-t\exp\left(-c\sigma^{-2}\delta^2\right) - \exp\left(-C_\epsilon\sigma^2t\right). 
\label{eq:FirstOrderErgodic3} 
\end{equation} 
\begin{proof}
We first prove 
\begin{equation}
\pr\left(\norm*{\frac{1}{\sigma^2t}\left(\overline{\pi}_{t}^\sigma-\overline{\pi}_t^0\right)- P_*\left(\mathcal{V}\right) }>\epsilon,T_{\sigma}<\tau_\delta\right)\,\le\,\exp\left(-C\sigma^2t\right), 
\label{eq:FirstOrderErgodic2} 
\end{equation} 
from which \eqref{eq:FirstOrderErgodic3}~follows just as in the proof of Theorem \ref{thm:ApproximateErgodic}~at the end of Section \ref{sec:Proof}. 
Noting that $D\pi(x)V(x)$ is constant with respect to $x\in B(\Gamma)$. 
Hence, using \eqref{eq:overlinepi}~and the It{\^o}~formula \eqref{eq:piItoRecentered}~for $\pi$, we have 
\[
\begin{aligned}
&\pr\left(\norm*{\frac{1}{\sigma^2 t}\left(\overline{\pi}_t^\sigma-\overline{\pi}_t^0\right)-P_*(\mathcal{V})}>\epsilon,t<\tau_\delta\right)\\
&\qquad\le\,\pr\left(\norm*{\frac{1}{2t}\int_0^{t}\sum_{k\in\N}D^2\pi(X_s)\left[B(X_s)e_k,B(X_s)e_k\right]\,ds-P_*(\mathcal{V})}>\epsilon/2,t<\tau_\delta\right) \\
&\qquad\qquad\quad+ \pr\left(\norm*{\frac{1}{\sigma t}\int_0^{t}D\pi(X_s)B(X_s)\,dW_s}>\epsilon/2,t<\tau_\delta\right) \\ 
&\eqqcolon\,\pr(F_1) + \pr(F_2). 
\end{aligned}
\]
Then, observe that 
\[
\begin{aligned}
\pr(F_1)\,&\le\,\pr\left(\norm*{\frac{1}{t}\int_0^{t}\frac{1}{2}\sum_{k\in\N}D^2\pi(X_s)[B(X_s)e_k,B(X_s)e_k] - \mathcal{V}(\pi_s)\,ds}>\epsilon/4,t<\tau_\delta\right) \\
&\qquad\quad + \pr\left(\norm*{\frac{1}{t}\int_0^{t}\mathcal{V}(\pi_s)\,ds - P_*(\mathcal{V})}>\epsilon/4,t<\tau_\delta\right)\\
&\eqqcolon\,\pr(F_{1,1})+\pr(F_{1,2}). 
\end{aligned}
\]
By the Lipschitz property of the vector field in \eqref{eq: assumption continuity vector field}, we have that there is a constant $C_{\pi} > 0$,
\[
\norm*{\frac{1}{t}\int_0^{t}\frac{1}{2}\sum_{k\in\N}D^2\pi(X_s)[B(X_s)e_k,B(X_s)e_k] - \mathcal{V}(\pi_s)\,ds}\,\le\,C_\pi\sup_{s\in[0,t]}\norm*{v_s}\,\le\,C_\pi\delta, 
\]
the last inequality holding when $t<\tau_\delta$. 
Hence, so long as we take $\delta<\delta_\epsilon\coloneqq\epsilon/4C_\pi$, we have that $\pr(F_{1,1})=0$. 
Then, under the assumption that $\mathcal{V}$ is $C^2$, the bound on $\pr(F_{1,2})$ follows from the first part of this theorem with $g=\mathcal{V}$, for a possibly different value of $m>0$. 
The bound on $\pr(F_2)$ is in Lemma \ref{Lemma bound the B event}.
\end{proof}
\end{cor}

\begin{lemma}
\label{Lemma bound the B event}
There exists $K_{F_2}>0$ such that for all $t>0$, 
\[
\begin{aligned}
\pr(F_2)\,&\equiv\, \pr\left(\norm*{\frac{1}{\sigma t}\int_0^{t}D\pi(X_s)B(X_s)\,dW_s}>\epsilon/2,\,t<\tau_\delta\right)\,\\
&\leq\, \exp\left(-K_{F_2}t^3\right). 
\end{aligned}
\]
\end{lemma}

\begin{proof}
Recall that $\int_0^{t}D\pi(X_s)B(X_s)\,dW_s$ is an $\mathcal{S}$-valued continuous martingale. By the Dambins-Dubins-Schwarz theorem (see Karatzas and Shreve \cite{KS12}, Theorem 3.4.6), there exist $m$ $\R$-valued Brownian Motions $\tilde{W}^j(t)$ such that the $j$th component of $\int_0^{t}D\pi(X_s)B(X_s)\,dW_s$ can be pathwise-identified with $\tilde{W}^j\big( \zeta^j_t \big)$, where $\zeta^j_t$ is defined as the quadratic variation of the $j$th component of  $\int_0^{t}D\pi(X_s)B(X_s)\,dW_s$ at time $t\ge0$, 
\begin{equation} 
\zeta_t^j\,=\,\left\langle\left(\int_0^\cdot\pi(X_s)B(X_s)\,dW_s\right)_j\right\rangle_t. 
\end{equation}
Observing that 
\begin{equation}
\label{eq: uniform bound for pi quadratic variation}
\begin{aligned}
\zeta_t^j\,&\le\,\left\langle\int_0^\cdot\pi(X_s)B(X_s)\,dW_s\right\rangle_t\\
&=\, \int_0^t\sum_{k\in\N}\norm*{D\pi(X_s)B(X_s)e_k}^2\,ds\\
&\le\,tM_\pi M_B 
\end{aligned}
\end{equation}
for $t<\tau$ and setting $c\coloneqq M_\pi M_B$, we find 
\begin{equation}
\begin{aligned}
\pr\left(F_2\right)\,
&\leq \,\pr\left( \sup_{0\leq s \leq t}\sup_{1\leq j \leq m}\big|\tilde{W}^j\big(\zeta_s^j \big) \big|
>\frac{\sigma\epsilon t}{2m},t<\tau\right) \\
&\leq \,\pr\left( \sup_{0\le s\le t}\sup_{1\leq j \leq m}\big|\tilde{W}^j( cs ) \big|
>\frac{\sigma\epsilon t}{2m},t<\tau\right). 
\end{aligned}
\end{equation}
Now, for any constant $k > 0$ and $0\le s\le t$,  
\begin{equation}
\begin{aligned}
\pr\left(
\tilde{W}^j( cs ) 
>\frac{\sigma\epsilon t}{2m},t<\tau\right) 
\,&= \,\pr\left( \sup_{0\leq s \leq t} \exp\left( k \tilde{W}^j(cs) \right) > \exp\left( k\frac{\sigma\epsilon t}{2m} \right),t<\tau\right)\\
&\leq \,\ExpOp\left[\exp\left( k \tilde{W}^j(ct)- k\frac{\sigma\epsilon t}{2m} \right) \chi\lbrace t<\tau \rbrace\right] \\
&\leq \,\exp\left( \frac{k^2}{2ct} -  \frac{k\sigma\epsilon t}{2m} \right) \label{eq: last bound}
\end{aligned}
\end{equation}
using Doob's submartingale inequality in the penultimate line (noting that $\exp\big( k \tilde{W}^j(ct) \big)$ is a submartingale, since $\exp$ is a convex function), and using the fact that $\tilde{W}^j(ct)$ is normally distributed with variance $ct$ in the last line. To optimize the bound in \eqref{eq: last bound}, choose 
\begin{equation}
k\,=\,\frac{c\sigma\epsilon}{2m}t^2. 
\end{equation}
After substitution, this yields the bound 
\begin{equation}
\pr\left(
\tilde{W}^j( cs ) 
>\frac{\sigma\epsilon t}{2m},t<\tau\right) 
\,\le\,\exp\left(-\frac{c(\sigma\epsilon)^2}{8m^2}t^3\right)\,\eqqcolon\, \exp\left(-K_{F_2}t^3\right). 
\label{eq: not last bound}
\end{equation}
\end{proof}

\begin{appendices}
\section{Closeness of the Isochronal and Variational Phases}
\label{sec:IsoVar}
Let $m'\le m$ denote the intrinsic dimension of $\mathcal{S}$ (and therefore $\Gamma$), and for $\alpha\in\mathcal{S}$ let $\{\alpha_i\}_{i=1}^{m'}$ denote a local coordinate system at $\alpha$.
We recall the $\delta$ neigbourhood of the manifold $\Gamma$, 
\begin{equation}
\Gamma_{\delta} = \big\lbrace x \in H: \norm{x-u} \leq \delta \text{ for some }u\in \Gamma \big\rbrace .
\end{equation}The variational phase is defined \cite{M22}~to be the map $\beta :\Gamma_\delta\rightarrow\mathcal{S}$ defined for sufficiently small $\delta>0$ for $x\in\Gamma_\delta$ as the unique point $\beta(x)\in\mathcal{S}$ such that 
\begin{equation}
\label{eq:Palpha}
P_{\beta(x)}[x-\gamma_{\beta(x)}]\,=\,0. 
\end{equation}
It is well-known from the theory of orbital stability that the variational phase is well-defined as long as $\delta$ is sufficiently small \cite[Lemma 4.3.3]{KP13}. We note this in the following Lemma.
\begin{lemma}
There exists $\hat{\delta}$ such that for all $\delta \leq \hat{\delta}$, and all $x\in \Gamma_\delta$, there exists a unique $\beta(x)$ satisfying \eqref{eq:Palpha}. 
\end{lemma}

\begin{lemma}
\label{lemma:IsoVar}
There exists $\delta_*>0$ such that for $\delta\in(0,\delta_*)$, if $\norm*{x-\gamma_{\pi(x)}}\le\delta$ then 
\[
\norm*{\gamma_{\pi(x)}-\gamma_{\beta(x)}}\,\le\,\delta^2. 
\]
\end{lemma}

\begin{proof}
It is well-established in the theory of orbital stability that for all sufficiently small $\delta$ and all $x\in \Gamma_{\delta}$,
\begin{align}
\norm*{\pi(x) - \beta(x)}  \leq \text{Const}\norm*{ x - \gamma_{\beta(x)} }^2. 
\end{align}
See the proof of Theorem 4.3.5 in \cite{KP13} for more details. Since the map $\beta \to \gamma_{\beta}$ is Lipschitz, we have proved the lemma. 
\end{proof}

By using Lemma \ref{lemma:IsoVar}, we are able to directly apply Theorem 5.1 of \cite{M22}~to prove \ref{lemma:LDT2}.

\begin{proof}[Proof of Lemma \ref{lemma:LDT2}]
Let $\beta_t\coloneqq\beta(X_t)$ be the variational phase of $X_t$. 
Then, by a union of events bound we have 
\[
\begin{aligned}
\pr[t>\tau_\delta]\,&=\,\pr\left[\sup_{s\in[0,t]}\norm*{X_s-\gamma_{\pi_s}}>\delta\right]\\
&\le\,\pr\left[\sup_{s\in[0,t]}\norm*{X_s-\gamma_{\beta_s}}>\frac{\delta}{2}\right] + \pr\left[\sup_{s\in[0,t]}\norm*{\gamma_{\beta_s}-\gamma_{\pi_s}}>\frac{\delta}{2}\right]. 
\end{aligned}
\]
By Lemma \ref{lemma:IsoVar}, we know that so long as $t<\tau_\delta$ we have  
\[
\begin{aligned}
 \pr\left[\sup_{s\in[0,t]}\norm*{\gamma_{\beta_s}-\gamma_{\pi_s}}>\frac{\delta}{2}\right]\,&\le\,
 \pr\left[\sup_{s\in[0,t]}\norm*{X_s-\gamma_{\pi_s}}^2>\frac{\delta}{2}\right]\\
&\le\,\pr\left[\delta^2>\frac{\delta}{2}\right]. 
\end{aligned}
\]
Taking $\delta$ sufficiently small, we thus obtan 
\[
\pr[t>\tau_\delta]\,\le\,\pr\left[\sup_{s\in[0,t]}\norm*{X_s-\gamma_{\beta_s}}>\frac{\delta}{2}\right]. 
\]
Applying Theorem 5.1 of \cite{M22}~completes the proof. 
\end{proof}

\end{appendices}

\subsection*{Acknowledgements} 
Z.P.A.~would like to thank J{\"u}rgen Jost for his continuing patience and support. 
This work was funded in part by the International Max Planck Research School for Mathematics in the Sciences.

\bibliographystyle{plain}
\bibliography{bibliography}{}

\end{document}